%% file: NewArxiv.tex
\documentclass{article}
\input{fatma-packages}
\input{fatma-macros}

\usepackage{thm-restate}
\usepackage{tikz-cd}
\usepackage{mathtools}
\usepackage{xcolor}
\usepackage{caption}
\usepackage{subcaption}

\PassOptionsToPackage{numbers, compress}{natbib}

\usepackage{natbib}
 \bibpunct[, ]{(}{)}{,}{a}{}{,}%

\def\true{1}

\def\flagJournal{0} 
 
\begin{document}

\title{Online Convex Optimization Perspective for Learning from Dynamically Revealed Preferences}
\author[1]{Violet (Xinying) Chen}
\author[1]{Fatma K{\i}l{\i}n\c{c}-Karzan}
\affil[1]{Tepper School of Business, Carnegie Mellon University, Pittsburgh, PA 15232}
\date{3 June 2021}

\maketitle
\begin{abstract}
We study the problem of online learning (OL) from revealed preferences: a learner wishes to learn a non-strategic agent's private utility function through observing the agent's utility-maximizing actions in a changing environment. We adopt an online inverse optimization setup, where the learner observes a stream of agent's actions in an online fashion and the learning performance is measured by regret associated with a loss function. We first characterize a special but broad class of agent's utility functions, then utilize this structure in designing a new convex loss function. We establish that the regret with respect to our new loss function also bounds the regret with respect to all other usual loss functions in the literature. 
This allows us to design a flexible OL framework that enables a unified treatment of loss functions and supports a variety of online convex optimization algorithms. We demonstrate with theoretical and empirical evidence that our framework based on the new loss function (in particular online Mirror Descent) has significant advantages in terms of regret performance and solution time over other OL algorithms from the literature and bypasses the previous technical assumptions as well. 
\end{abstract}

\section{Introduction} \label{sec:intro}
Preferences of an agent implicitly dictates his/her actions, and influence for example what a company should offer as its products or how a company should personalize recommendations to an individual customer (agent). This creates incentives for the company/central decision maker to learn the preferences of their agents. 
Nevertheless, in reality, the true preferences of the agents are often private to the individual agents and are only implicitly revealed in the form of their behaviors/actions to the central decision maker. Such typical interactions for example include a streaming platform suggesting a number of videos to a user and tracking whether the user watches or likes them. 
As evident from such scenarios, inferring the agents' preference information through agent interactions and observations of their behaviors is a critical task for the decision makers in such settings.

A common assumption adopted to formalize the problem of learning from revealed preferences is that rational agents are \emph{utility} maximizers, that is, they choose actions to maximize their utility functions subject to a set of restrictions. The central decision maker interacting with the agents is the \emph{learner}. An important learner-centric goal is to design schemes for the learner to extract useful information on the agents' utility functions. This learning point of view of revealed preferences has been explored in a broad range of literature from economics (e.g., \cite{varian2006revealed, BeigmanV2006}), machine learning (e.g., \cite{balcan2014learning,dong2018strategic}) and operations research (e.g., \cite{BarmannMPS2018, Kuhn2018}). Based on the type of learner-agent interactions and information feedback, such preference learning schemes vary in information requirement, preference elicitation objective and learning complexity.

In this paper, we focus on a specific setup where the learner seeks to learn the utility function of a non-strategic agent while receiving information about the agent's actions in an online fashion.  This setup fits naturally in applications where the agents benefit from effective learning of their true preferences. For example this is the case when a streaming platform interacts with its users to learn their preferences.
In this example, in a typical interaction, the platform recommends videos to a user and the user takes actions based on the recommendations. User actions, e.g., clicks, movie streaming, etc., are fully observable to the platform and they reflect the user's true preferences. 

\subsection{Related Literature}  \label{sec:literature}
\cite{varian2006revealed} is one of the earliest and most celebrated work for learning from revealed preferences in the economics literature. They study constructing utility functions of the agent to explain a sequence of her/his observed actions. Nevertheless, this approach has a main shortcoming---a utility function capable of explaining past actions not necessarily also guarantees accurate predictions of the future actions. Consequently, \cite{BeigmanV2006} have initiated a new line of research to learn utility functions capable of predicting future actions with statistical performance guarantees. \cite{BeigmanV2006} examine a statistical setup where the learning algorithm takes as input a batch of observations and is evaluated by its sample complexity guarantees. 
\cite{zadimoghaddam2012} focus on the setting where the agent has a linear or linearly separable concave utility function, and propose learning algorithms with polynomially bounded sample complexity. \cite{balcan2014learning} identify a connection between the problem of learning a utility function and the structured prediction problem of D-dimensional linear classes. Through this connection, \cite{balcan2014learning} suggest an algorithm for learning utility functions that is superior (in terms of sample complexity) than the method from \cite{zadimoghaddam2012} in the case of linear utility functions and is also applicable for learning separable piecewise-linear concave functions and CES functions with explicit sample complexity bounds. 

As an alternative to this statistical view, \cite{balcan2014learning} study a query-based learning model, where the learner aims to recover the exact utility function by querying an oracle for the agent's optimal actions. The query-based models consider an online feedback mechanism where the learner receives one observation of the agent's action at a time. When the learner has the power to choose which observation to receive from the query oracle, \cite{balcan2014learning} give exact learning algorithms for several classes of utility functions. There is a recent research stream on \emph{learning to optimize} the learner's objective function based on information from revealed preferences of the agents. In this stream it is often assumed that the learner has similar power on the selection of the observations.  For example, \cite{amin2015online} and \cite{ji2018social} propose algorithms for finding the profit-maximizing prices for a seller, who has price controlling power and learns buyer preferences by observing the buying behavior at different price levels.
\cite{roth2016watch} and \cite{dong2018strategic} consider the learning task as Stackelberg games, where the leader player is the learner and a follower player is a \emph{strategic} agent with incentive to manipulate actions and hide information. More specifically, \cite{roth2016watch} study the price setting problem with unknown follower preferences in a Stackelberg game. \cite{dong2018strategic} consider preference learning in strategic classification, where the leader player releases classifiers to strategic follower players. When the leader's classifier selection problem is a convex program, \cite{dong2018strategic} provide an online zeroth-order optimization  algorithm for minimizing the leader's Stackelberg regret.

We note two restrictions with the problem setup in these fore-mentioned papers. First, the assumption that the learner can choose observations is not always achievable in practice. A more realistic setup is accomodated by the data-driven inverse optimization view where the learner does not control the sequence of observations.
Second, when the learner is optimizing an objective function that does not explicitly measure how well s/he is learning about the agent, the approaches that are effective for choosing the learner's objective-optimizing action provide no guarantees on the quality of the learned agent information. 

Inverse optimization generalizes the query-based view and offers a natural abstraction of learning from revealed preferences. This approach is typically used in settings with non-strategic agents, in which the agents have no incentive to hide information from the learner, and thus an agent's decisions reveal her/his true preferences.
In this setting, the learner's goal is to recover unknown parameters of an agent's utility function from the observations of her/his true optimal solution.
Early studies on inverse optimization examine the setting where the agent's optimization problem is fixed, see e.g., \cite{ahuja2001inverse,Iyengar2005,heuberger2004inverse,schaefer2009inverse}. 
Unfortunately, this classical setup is limited in its practical applicability as it ignores uncertainty in the environment.
A new thread of research on data-driven inverse optimization studies a flexible setup, where the learner observes the agent's optimal or sub-optimal decisions corresponding to varying external data signals.
In the noiseless case, that is, when observations of optimal solutions/agent actions are available, \cite{KeshavarzWB2011} show that data-driven inverse optimization of convex programs is polynomial time solvable. In the case of noisy observations, \cite{AswaniSS2018} proves that such problems are NP-hard in general.

Data-driven inverse optimization is further categorized based on whether observations are given as a batch upfront or in an online manner. 
In the batch setup, \cite{KeshavarzWB2011} study the inverse optimization of identifying the unknown affine weights in a convex objective function that is an affine combination of pre-selected basis convex functions. 
Recent work of \cite{AswaniSS2018} and \cite{Kuhn2018} in the batch setup investigates the inverse optimization of general convex programs without the basis function structure. \cite{AswaniSS2018} adopt the  \emph{prediction loss} $\ell^{pre}$, which measures the difference between the observed agent action and the predicted agent action through squared norm distance, as the inverse optimization objective. They formulate the inverse problem into a bilevel program using Lagrangian duality, and present two heuristic algorithms with approximation guarantees for solving the bilevel formulation. \cite{Kuhn2018} use \emph{suboptimality loss} $\ell^{sub}$, which is defined as the difference between objective values at the observed agent action and the predicted action, as their loss function and provide a distributionally robust formulation of the inverse problem. Batch setup requires that the learner receives observations of the agent's actions all at once. However, obtaining a large batch of observations all at once as well as learning from such a batch often presents operational and computational challenges. In practice, such strong batch feedback is rare as the learner often interacts with the agent repetitively in a  dynamic environment. 

A recent stream of research \cite{BarmannMPS2018,DongCZ2018} adopts a dynamic information acquisition setup and studies the online data-driven inverse optimization where the learner observes a stream of the agent's actions one by one in an online fashion. Both \cite{BarmannMPS2018} and \cite{DongCZ2018} suggest OL algorithms and measure their performance via the \emph{regret}, i.e., the difference between the losses incurred from online estimates and the offline optimal estimate. \cite{BarmannMPS2018} consider the problem of learning the linear utility function of an agent given the noiseless online observations of the agent's actions in a dynamic environment. They propose two specialized OL algorithms with first-order oracles that both achieve a bound of $O(\sqrt{T})$ on the sum of the suboptimality loss $\ell^{sub}$ and the estimate loss $\ell^{est}$ after $T$ periods but lacks regret guarantees with respect to the prediction loss $\ell^{pre}$. 
\cite{DongCZ2018} consider the setup, where the learner wishes to learn an unknown linear component of an agent's quadratic objective function from noisy observations. By utilizing the implicit OL framework of \cite{Kulis2010} equipped with a Mixed Integer Second Order Cone Program (MISOCP)-based solution oracle, they provide a regret bound of $O(\sqrt{T})$ with respect to the prediction loss $\ell^{pre}$ after $T$ periods whenever $\ell^{pre}$ is convex. \cite{DongCZ2018} present a number of rather technical assumptions that guarantee the convexity of $\ell^{pre}$, however these assumptions are not only difficult to verify but also quite restrictive. In fact, in \cite{DongCZ2018}, these were shown to hold only for a very specific case of a specific convex quadratic problem class.

\subsection{Contributions and Outline} \label{sec:outline}
We follow an online data-driven inverse optimization perspective for learning from dynamically revealed preferences. Our contributions along with an outline of the paper is as follows.
\begin{itemize}
\item We present a formal description of our problem setting in Section \ref{sec:setting}. In our setup, the learner monitors a sequence of data signals and observes the respective rational decisions of a non-strategic agent without noise over a finite time horizon of $T$ time steps. 
The learner operates and receives information in an online fashion, and updates an estimate $\theta_t$ of $\theta_{true}$ using newly available information at each time step.  
Section~\ref{sec:forwardprob} introduces the agent's problem and discusses 
a rather broad decomposable structure assumption on the agent's utility functions that is capable of representing all of the utility functions studied in the data-driven inverse optimization literature as well as  other key utility functions. Section~\ref{sec:inverse_problem} describes the learner's inverse optimization problem that minimizes a given \emph{loss function} $\ell(\cdot)$ to obtain an accurate estimate $\theta$ of the hidden parameter $\theta_{true}$.  
In Section~\ref{sec:inverseprob}, we utilize our structural assumption on the agent's utility function to design a new convex loss function, namely \emph{simple loss} $\ell^{sim}$. 
We establish in Section~\ref{sec:pi} that in the noiseless setting, a bounded regret with respect to $\ell^{sim}$ also guarantees a bounded regret with respect to all the other loss functions; see Proposition~\ref{prop:pi} and Corollary~\ref{cor:strconv}. 

\item Convexity and simplicity of $\ell^{sim}$ enables us to use an online convex optimization (OCO) framework (see  Section~\ref{sec:OL}) 
that offers the flexibility to use different OL algorithms, such as, online Mirror Descent (MD) utilizing a first-order oracle (Section~\ref{sec:FOL}) and implicit OL based on a solution oracle (Section~\ref{sec:implicitOL}). 
In the noiseless setup, our framework equipped with online MD covers \emph{all} of the problem classes studied in the online data-driven inverse optimization literature, and matches the corresponding state-of-the-art regret bounds with respect to \emph{all} of the loss functions in a \emph{unified} manner. In particular, our results immediately generalize the customized algorithms from \cite{BarmannMPS2018} and completely bypass the requirement to verify the rather technical assumptions of \cite{DongCZ2018} and the need to use their expensive MISOCP-based solution oracle; see Section~\ref{sec:litcomp} for a detailed comparison discussion. 

\item Our numerical study in Section~\ref{sec:app} highlights that when compared to the $\ell^{pre}$-based implicit OL with an MISOCP solution oracle approach of \cite{DongCZ2018}, $\ell^{sim}$-based OL algorithms equipped with a first-order oracle or a solution oracle, particularly online MD, demonstrate significant advantages in terms of both the learning performance (i.e., regret bounds) and the computation time. This is directly in line with our theoretical results. Moreover, these results seem to be fairly robust with respect to changes in the structure of the agent's domain as well as the noise in observations. 
\end{itemize}

All proofs are given in Appendix~\ref{sec:proof}, and derivation details on the solution oracles are presented in Appendix~\ref{sec:app-sol}.

\textbf{Notation.} 
We let $\R_+^n$ be the set of nonnegative $n$-dimensional vectors. For a given vector ${v}$, we use $v_i$ to denote its $i$-th element. 
We let $[n] \coloneqq \{1, \ldots, n\}$, and we use $\{a_i\}_{i \in [n]}$ to represent a collection of entries, such as vectors, functions, etc., indexed with $i \in [n]$. 
For a differentiable function $f$, we use $\nabla f({x})$ to denote the gradient of $f$ at ${x}$. For a nondifferentiable function $f$, we use $\partial f({x})$ to denote the subdifferential of $f$ at ${x}$. 

\section{Problem Setting} \label{sec:setting}
In our setting, the learner monitors a sequence of external signals $\{{u}_t\}_{t \in [T]}\subseteq\R^k$ and observations $\{{y}_t\}_{t \in [T]}\subseteq\R^n$ of the agent's respective optimal decisions over a finite time horizon of $T$ time steps. 

\subsection{Forward Problem} \label{sec:forwardprob}
For a fixed exogenous signal ${u}$, the agent's optimal decision ${x}({\theta}_{true}; {u})$ is given by the \emph{forward problem}:
\vspace{-5pt}
\begin{equation} \label{sys:forward}
{x}({\theta}_{true}; {u})\in\argmin_{{x}} \left\{f({x};{\theta}_{true},{u}):~ g({x};{u}) \leq 0,~ {x} \in \CX\right\},
\vspace{-5pt}
\end{equation}
where $\CX\subseteq\R^n$ is the static domain of the agent's problem and ${\theta}_{true}\in\R^p$ is a parameter known only by the agent, $f$ represents the negative of the agent’s utility function capturing her/his preferences, and $g$ is a (set of) constraint(s) defining agent’s feasible actions in $\CX$. 
We study a special class of objective functions $f$ in the forward problem \eqref{sys:forward}. 
\begin{assumption} \label{assp:fstr}
The function $f$ has a decomposable structure of the form 
$f({x}; \theta, {u}) = f_1({x};{u}) + f_2(\theta;{u}) + \la \theta, c({x}) \ra$, where $c({x}) = \left(c_1({x}), \ldots, c_p%
({x}) \right)$. 
\end{assumption}

While Assumption~\ref{assp:fstr} may appear to be restrictive, it still allows us to capture all problem classes studied in the literature as well as two other important classes of utility functions that have not been addressed in the online preference learning setting: CES (constant elasticity of substitute) function, and Cobb-Douglas function, which is a limit case of the CES function. See Appendix~\ref{app:utilfunc} for the transformations of CES and Cobb-Douglas functions to satisfy Assumption \ref{assp:fstr}. 
Unfortunately, the other well-studied limit case known as the Leontief function does not fit into the same framework. 
In Section \ref{sec:inverseprob}, we will show how Assumption~\ref{assp:fstr} is advantageous in establishing desirable convexity of the suboptimality loss function used in this literature, and designing a new convex loss function. 

\begin{remark}
Assumption~\ref{assp:fstr} allows for the possibility of the function $c$ to obscure information. In most examples of interest in the online inverse optimization or preference learning, the function $c$ will provide direct information on $x$, such as $c_i(x_i)=x_i$, or $c_i(x_i)=x_i^2$, etc. That said, one can purposefully select this function $c$ to obscure information on $x$, e.g., $c_i(x_i)=0$ for almost all values of $x_i$. In such cases, our framework as well as any other meaningful approach may fail to provide interesting guarantees (i.e., sublinear regret bounds) in an online setup. We further discuss this in Appendix~\ref{app:cstruc}.
\end{remark}

\subsection{Inverse Problem}\label{sec:inverse_problem}
Given a signal $u$, under \emph{perfect information} the learner observes the agent's optimal solution without noise, i.e., ${y} = {x}(\theta_{true};{u})$; under \emph{imperfect information}, ${y} = {x}(\theta_{true};{u}) + \epsilon$, where $\epsilon \in \R^n$ is the noise that the learner suffers from when observing agent's action. Consistent with the literature, we assume that the learner has access to an \emph{agent response oracle}, with which the learner can generate the \emph{predicted action} ${x}(\theta;{u})$ at a given signal $u$ and estimate $\theta$ by solving the following model obtained from \eqref{sys:forward} where  $\theta_{true}$ is replaced with $\theta$  
\begin{align} \label{sys:forward-theta}
{x}(\theta;{u})\in\argmin_{{x}} \left\{f({x};{\theta},{u}):~ g({x};{u}) \leq 0,~ {x} \in \CX\right\}.
\end{align}
Based on external signals, past observations of the agent's respective optimal decisions, 
and the knowledge of a convex set $\Theta$ containing $\theta_{true}$, the learner wishes to predict $\theta$ values that will mimic closely the agent's preference-driven action $x(\theta_{true};u)$. The performance of learner's estimates are measured via \emph{loss function} $\ell$: given signal ${u}$, the learner incurs  $\ell(\theta,{x}(\theta;{u});{y},{u})$ as the loss for the estimate $\theta$, where ${y}$ is the learner's observation of ${x}(\theta_{true}; {u})$ and ${x}(\theta;{u})$ is the learner's prediction of the agent's optimal decision as defined in \eqref{sys:forward-theta}. Specifically, $\ell: \R^p \times \R^n \mapsto \R$ takes $(\theta,{x}(\theta;{u})) \in \R^p \times \R^n$ as independent variables and $({y},{u})$ as given parameters. We refer to the learner's loss minimization problem as the \textit{inverse problem}. More formally, given the revealed parameters ${u}$ and ${y}$, this \emph{inverse problem} is a bilevel program of the form
\begin{align}\label{sys:inverse}
    \min_{\theta,{x}(\theta;{u})} \left\{ \ell(\theta,{x}(\theta;{u}); {y},{u}):~ {x}(\theta;{u}) \in \text{argmin}_{{x} \in \mathcal{X}} \{f({x};\theta,{u}):
    g({x};{u}) \leq 0 \}, ~ \theta \in \Theta  \right\}.
\end{align}

In the online inverse optimization over a finite time horizon $T$, at each time step $t \in [T]$, the learner generates an estimate $ \theta_{t} \in \Theta$ for the true parameter $\theta_{true}$ using the current signal $u_t$, the past information $\{(y_{t'},u_{t'})\}_{t'\in[t-1]}$, and the feedback collected on the loss functions $\ell_{t'}(\theta) \coloneqq \ell(\theta, {x}(\theta;{u}_{t'}); {y}_{t'}, {u}_{t'})$, $t' \in [t-1]$, i.e., from the previous $t-1$ time steps, and then the learner observes $y_t$. Typical OL algorithms rely on the feedback on the current loss function $\ell_t(\theta)$ such as the first-order information, i.e., the gradient $\nabla \ell_t(\theta_t)$ or a subgradient $\partial \ell_t(\theta_t)$ (indeed we will operate with this type of feedback too). In the case of online inverse optimization, once the observation $y_t$ is revealed, we must demonstrate that such feedback for $\ell_t(\theta)$ is possible based on the information available to the learner at the current iteration $t$, i.e., $\{(y_{t'},u_{t'},\theta_{t'})\}_{t'\in[t]}$, and we will illustrate that for our choice of loss function $\ell_t(\theta)$ this is indeed possible.
The goal of the online decision maker is to minimize the cumulative loss $\sum_{t \in [T]} \ell_t(\theta_t)$. The performance of an OL algorithm is measured via \emph{regret}, that is, the difference between the cumulative loss incurred from the online decisions $\{\theta_t\}_{t \in [T]}$ and the best fixed decision in hindsight:
\begin{equation} \label{defeq:regret}
R_T\left(\{\ell_t\}_{t \in [T]}, \{\theta_t\}_{t \in [T]}\right) \coloneqq \textstyle \sum_{t \in [T]} \ell_t(\theta_t) - \min_{\theta \in \Theta} \sum_{t \in [T]} \ell_t(\theta).
\end{equation}

\begin{remark}
In standard OL, specially OCO where $\{\ell_t(\theta)\}_{t\in[T]}$ are convex for all $t$, loss functions follow an adversarial view, i.e., an adversary who tries to hide as much information as possible from the learner generates them and then they are revealed to the learner. In contrast to this standard adversarial view in OCO, in our inverse learning application, we are concurrently designing our loss functions $\{\ell_t(\theta)\}_{t\in[T]}$ to which we will apply deterministic OCO algorithms as well. This may bring the question of whether the usual guarantees of OCO regret minimizing algorithms will remain valid in our setup or not. 
To this end, we highlight that the regret guarantees provided by the standard OCO algorithms hold for \emph{arbitrary} families of convex loss functions $\{\ell_t(\theta)\}_{t\in[T]}$. This flexibility in handling loss functions, often adversarial ones, makes OCO a useful and versatile tool in many applications (see \cite{hazan2019}).  
In fact, the application of OCO to \emph{non-arbitrary} functions as a tool has been employed previously in the context of designing OCO-based frameworks for solving robust convex optimization in \cite{BenTalHazan2015,Ho-NguyenKK16RO,ho2019exploiting}. 
For our particular application, in Section \ref{sec:inverseprob}, we design linear functions $\{\ell^{sim}_t(\theta)\}_{t\in[T]}$ to which we apply OCO algorithms. 
While using non-arbitrary classes of loss functions does not invalidate any guarantees from deterministic OCO methods such as online MD, more caution is needed for OCO algorithms which involve randomness and provide guarantees on \emph{expected regret} such as stochastic gradient descent. This is due to the fact that the design of the loss functions in specific applications may create undesirable dependence among random variables and invalidate certain steps used in the analysis of stochastic OCO algorithms. Hence, here we will focus on deterministic OCO algorithms. 
\end{remark}

\subsection{Loss Functions}
\label{sec:inverseprob}
Loss function $\ell(\theta)$ plays a key role in the formulation of the inverse problem \eqref{sys:inverse}. The following are common loss functions used in inverse optimization context (recall that $f$ is the agent's forward objective in \eqref{sys:forward} and $x(\theta; u_t)$ is the optimal solution to \eqref{sys:forward-theta} for given $\theta$, $u_t$): 
\begin{itemize}
    \item Prediction loss: $\ell^{pre}(\theta, {x}(\theta;{u}_t); {y}_t, {u}_t) \coloneqq \norm{{y}_t - {x}(\theta;{u}_t)}^2$,
    \item Suboptimality loss: $\ell^{sub}(\theta, {x}(\theta;{u}_t); {y}_t, {u}_t) \coloneqq f({y}_t; \theta, {u}_t) - f({x}(\theta;{u}_t); \theta, {u}_t)$, and
    \item Estimate loss: $\ell^{est}(\theta, {x}(\theta;{u}_t); {y}_t, {u}_t) \coloneqq  f({x}(\theta;{u}_t); \theta_{true}, {u}_t) - f({y}_t; \theta_{true}, {u}_t)$.
\end{itemize}
Within data-driven inverse optimization in a batch setup, $\ell^{pre}$ is used in \cite{AswaniSS2018} and $\ell^{sub}$ is used in \cite{Kuhn2018} (see Section \ref{sec:literature}). In the online inverse optimization setup, $\ell^{sub}$ and $\ell^{est}$ are studied by \cite{BarmannMPS2018} under the assumption that $f$ is linear in $x$, and $\ell^{pre}$ by \cite{DongCZ2018} when the forward problem is a quadratic program with a special structure. The OL algorithms from these latter two papers are customized for the chosen loss functions and forward problem structure, indicating the lack of a unified general framework. 

We introduce the following shorthand notation.
\[
\ell^{pre}_t(\theta) \coloneqq \ell^{pre}(\theta, {x}(\theta;{u}_t); {y}_t, {u}_t), ~\ell^{sub}_t(\theta) \coloneqq \ell^{sub}(\theta, {x}(\theta;{u}_t); {y}_t, {u}_t),~
\ell^{est}_t(\theta) \coloneqq \ell^{est}(\theta, {x}(\theta;{u}_t); {y}_t, {u}_t).
\vspace{-5pt}
\]

We first establish that under Assumption \ref{assp:fstr}, $\ell^{sub}(\theta)$ becomes a convex function of $\theta$. To the contrary, $\ell^{pre}_t(\theta)$ and $\ell^{est}_t(\theta)$ are not guaranteed to be convex; see Appendix \ref{app: convexity}. 
\begin{lemma}\label{lem:l-sub-characterization}
Under Assumption~\ref{assp:fstr}, $\ell^{sub}_t(\theta)$ is convex in $\theta$ for every $u_t,y_t$ and $t \in [T]$.
\end{lemma}

We next utilize the structure of $f$ in Assumption \ref{assp:fstr} to design a new loss function. 
\begin{definition} \label{def:lsim}
Suppose Assumption \ref{assp:fstr} holds. We define the \emph{simple loss} as 
\[\ell^{sim}(\theta, {x}(\theta_t;{u}_t); {y}_t, {u}_t) \coloneqq
    \la \theta, c({y}_t) - c({x}(\theta_t;{u}_t)) \ra +  \la \theta_{true}, c({x}(\theta_t;{u}_t)) - c({y}_t)\ra.
    \qedhere
\]
\end{definition}

Let $\ell^{sim}_t(\theta) \coloneqq \ell^{sim}(\theta, {x}(\theta_t;{u}_t), {y}_t, {u}_t).$ In $\ell^{sim}_t(\theta)$, the term ${x}(\theta_t; {u}_t)$ is the optimal solution to \eqref{sys:forward-theta} with given $\theta_t$ and $u_t$, and it can be viewed as a prediction of the agent's action at the current estimate $\theta_t$ of the true parameter $\theta_{true}$. Hence, when $\ell^{sim}$ is used as the loss function in online inverse optimization, at each time step $t$, $\ell^{sim}_t(\theta)$ has an explicit dependence on the revealed signal $u_t$, the observation $y_t$ of agent's true optimal action, and the predicted action $x(\theta_t; u_t)$ using the estimate $\theta_t$ generated based on the information from previous $t-1$ time steps and the signal $u_t$. Here, it is noteworthy to highlight that since $\theta_t$ is determined before the function $\ell^{sim}_t(\theta)$ is revealed, there is no cyclic dependence between them. 

Next, we note that $\ell^{sim}_t(\theta)$ is a convex function of $\theta$, which is important for its use in our online inverse optimization framework. We will demonstrate in Section \ref{sec:pi} that under Assumption \ref{assp:fstr}, regret minimization based on $\ell^{sim}(\theta)$ also leads to performance guarantees with respect to the all other loss functions. 

\begin{lemma} \label{lem:l-sim_convex}
Under Assumption~\ref{assp:fstr}, $\ell^{sim}_t(\theta)$ is linear (hence convex) in $\theta$ for every $t \in [T]$.
\end{lemma}

\subsection{Regret Performance Measures for Preference Learning} \label{sec:pi}
We will develop an OCO-based framework for preference learning. To this end,
we have already introduced a loss function, i.e., $\ell^{sim}$, that is convex under Assumption \ref{assp:fstr}. In this section, we show that in the perfect information setup (i.e., when there is no noise on the observations ${y}_t = {x}(\theta_{true},{u}_t)$ for all $t$), the regret with respect to $\ell^{sim}$ indeed bounds the regrets with respect to all other loss functions of interest as well. Although restrictive, perfect information setting is still relevant in practice in settings such as the streaming platform described in Introduction. For completeness, we discuss the \emph{imperfect information} case from both theoretical and empirical aspects (see Appendices~\ref{app:imperinfo}~and~\ref{app:imperexp}), and our preliminary study indicates the potential of further developments in this direction.

Our main result establishes a fundamental guarantee among the regret bounds with respect to $\ell_t^{sim}, \ell_t^{sub}$ and $\ell_t^{est}$. 

\begin{proposition} \label{prop:pi}
Suppose Assumption \ref{assp:fstr} holds and there is no noise on the observations. Then, for any sequence $\{\theta_t\}_{t \in [T]}$, we have  
\begin{align*}
(a)~ & R_T(\{\ell^{sub}_t\}_{t \in [T]}, \{\theta_t\}_{t \in [T]}),~ R_T(\{\ell^{est}_t\}_{t \in [T]}, \{\theta_t\}_{t \in [T]}),\text{ and } R_T(\{\ell^{pre}_t\}_{t \in [T]}, \{\theta_t\}_{t \in [T]}) \geq 0, \\
(b)~ & R_T(\{\ell^{sub}_t\}_{t \in [T]}, \{\theta_t\}_{t \in [T]}) + R_T(\{\ell^{est}_t\}_{t \in [T]}, \{\theta_t\}_{t \in [T]}) = \textstyle \sum_{t=1}^T \ell^{sim}_t(\theta_t), \\
(c)~ & R_T(\{\ell^{sim}_t\}_{t \in [T]}, \{\theta_t\}_{t \in [T]}) \geq R_T(\{\ell^{sub}_t\}_{t \in [T]}, \{\theta_t\}_{t \in [T]}) + R_T(\{\ell^{est}_t\}_{t \in [T]}, \{\theta_t\}_{t \in [T]}).
\end{align*}
As a consequence of (c), 
$R_T(\{\ell^{sim}_t\}_{t \in [T]}, \{\theta_t\}_{t \in [T]})$ upper bounds both $R_T(\{\ell^{sub}_t\}_{t \in [T]}, \{\theta_t\}_{t \in [T]})$ and $R_T(\{\ell^{est}_t\}_{t \in [T]}, \{\theta_t\}_{t \in [T]})$.
\end{proposition}

When $f$ is a strongly convex function in ${x}$, \cite[Proposition 2.5]{Kuhn2018} shows that $\ell^{sub}_t(\theta) \geq \frac{\gamma}{2}\ell^{pre}_t(\theta)$ for all $t$ and for all $\theta \in \Theta$, with $\gamma$ being the strong convexity parameter of $f$. Hence, this result enables us to derive a further regret bound for the loss functions $\{\ell_t^{pre}\}_{t\in[T]}$. 
\begin{corollary} \label{cor:strconv}
Suppose Assumption \ref{assp:fstr} holds and there is no noise. Assume further that $f$ is a strongly convex function of ${x}$ for every $\theta$, i.e., there exists $\gamma > 0$ such that $f({x};\theta,{u}) - f({y};\theta,{u}) \geq \langle s_{{y}}, {x}-{y}\rangle + \frac{\gamma}{2} \norm{{x} - {y}}^2$, where $s_{{y}}$ is a subgradient of $f({y};\theta,{u})$ with respect to ${y}$. Then, for any sequence $\{\theta_t\}_{t \in [T]}\subseteq \Theta$ we have $R_T(\{\ell^{sub}_t\}_{t \in [T]},\{\theta_t\}_{t \in [T]}) \geq \frac{\gamma}{2} R_T(\{\ell^{pre}_t\}_{t \in [T]},\{\theta_t\}_{t \in [T]})$.
\end{corollary}

Assumption \ref{assp:fstr} ensures that $\ell_t^{sim}$ is a convex function of $\theta$, and thus any deterministic OCO algorithm will be applicable for regret minimization with respect to $\{\ell^{sim}_t\}_{t\in[T]}$. Then, as a consequence of Proposition~\ref{prop:pi} (and Corollary~\ref{cor:strconv}), such algorithms will also be minimizing regret with respect to the loss functions $\ell_t^{sub}, \ell_t^{est}$ (and $\ell_t^{pre}$), as well. 

\begin{remark}
The regret bounds with respect to these loss functions have the following implications under perfect information. A sublinear regret bound with respect to $\ell^{sim}$ implies that the average loss incurred by the estimates $\{\theta_t\}$ approaches the offline optimal loss over time.Sublinear regret bounds with respect to $\ell^{sub}$ and $\ell^{est}$ indicate that the learner is able to generate estimates $\{\theta_t\}$ that lead to vanishing errors in the predicted agent's objective function values. A sublinear regret bound with respect to $\ell^{pre}$ additionally indicates that the average $\norm{\cdot}_2$-norm distance between the predicted agent's action and her/his true action decreases to zero over time.
Note that none of these regret guarantees in particular
ensures that the $\{\theta_t\}$ generated from the online learning process are good approximations of $\theta_{true}$. In general, this is an overly ambitious task as \cite[Example 3.2]{BarmannMPS2018} has shown a simple case where the exact recovery of $\theta_{true}$ cannot be guaranteed. We note that stronger performance guarantees, such as $\frac{1}{T}\sum_{t \in [T]} \norm{\theta_t - \theta_{true}} \rightarrow 0$ may be possible for special cases, for example, when $x(\theta;u_t)$ has a closed form expression as a continuous function in $\theta$. In addition, in certain cases the optimal actions from the forward problem may be non-unique, our framework is not aiming to predict the chosen action $x(\theta_{true}; u_t)$, instead, we measure the learning performance with regret values based on the objective value of the agent.
\ifx\flagJournal\true \epr \fi 
\end{remark}

The case when the learner has access to only imperfect information about the agent's actions is of natural interest as well. \cite{Kuhn2018} identify two types of noisy information as of interest: (i) \emph{measurement noise}, that is, for all $t\in[T]$, the learner observes $y_t = {x}(\theta_{true}, {u}_t) + \epsilon_t$ with $\epsilon_t$ denoting a random noise, and (ii) \emph{bounded rationality}, which means for all $t\in[T]$, the agent may choose a sub-optimal action instead of ${x}(\theta_{true}, {u}_t)$. 
Such imperfect information does not affect the convexity property of loss functions, and so both $\ell^{sub}$ and $\ell^{sim}$ remain convex (see Lemma~\ref{lem:l-sub-characterization}~and~\ref{lem:l-sim_convex}) still enabling the use of OCO algorithms for regret minimization with respect to these loss functions. However, since ${y}_t$ is no longer guaranteed to be a minimizer of \eqref{sys:forward} with $u = u_t$, Proposition \ref{prop:pi} does not hold in general, and consequently $R_T(\{\ell^{sim}_t\}_{t \in [T]}, \{\theta_t\}_{t \in [T]})$ is not guaranteed to bound the regrets with respect to the other loss functions. In addition, due to the noises in ${y}_t$, $R_T(\{\ell^{sub}_t\}_{t \in [T]}, \{\theta_t\}_{t \in [T]})$ can no longer accurately measure learning performance with respect to the agent's true objective values either.
See Appendix~\ref{app:imperinfo} for details on the consequences of imperfect information on regret.

\section{Online Learning Algorithms}\label{sec:OL}
Under Assumption \ref{assp:fstr}, both $\ell^{sim}(\theta)$ and $\ell^{sub}(\theta)$ are convex in $\theta$, so online inverse optimization with respect to either loss function can be done in an OCO framework. To take advantage of the unifying capability from $R_T(\{\ell^{sim}_t\}_{t \in [T]}, \{\theta_t\}_{t \in [T]})$, here we will focus on $\ell^{sim}$ as our loss function, leaving the investigation of $\ell^{sub}$-based methods for future work.

We equip our framework with two well-known deterministic OL regret minimization algorithms utilizing different oracles: Online \emph{Mirror Descent} (MD), which is a classical OCO algorithm that utilizes a first-order oracle, and the implicit OL algorithm introduced in \cite{Kulis2010} that is based on a solution oracle. Under Assumption~\ref{assp:fstr} and in the case of perfect information, both algorithms generate regret bounds with respect to $\ell^{sim}$, which then imply regret bounds with respect to the other loss functions as well (see Section~\ref{sec:inverseprob}).

For exposition convenience, we state all of these algorithms in the same online setup: the learner receives observations $\{{y}_t,{u}_t\}_{t \in [T]}$ and generates $\{\theta_t\}_{t \in [T]}\subseteq \Theta$ to minimize the regret $R_T(\{\ell_t\}_{t \in [T]}, \{\theta_t\}_{t \in [T]})$. 

\subsection{Online Convex Optimization with First-Order Oracle} \label{sec:FOL}
We review the well-known first-order OCO algorithm, namely the online \emph{Mirror Descent} (MD) algorithm in the proximal setup. We follow the presentation and notation of \cite{MD2011} and define the following standard components of the proximal setup:
\begin{itemize}
\item \emph{Norm}: $\|\cdot\|$ on the Euclidean space $\E$ where the domain $\Theta$ lives, along with its dual norm $\|\zeta\|_*:=\max\limits_{\|\theta\|\leq1}\langle\zeta,\theta\rangle$.
\item \emph{Distance-Generating Function} ({\dgf}): A function $\omega(
\theta):\Theta \rightarrow \R$, which is convex and continuous on $\Theta$, and admits a selection of subdifferential $\partial\omega(\theta)$ that is continuous on the set $\Theta^\circ:=\{\theta\in \Theta:\partial\omega(\theta)\neq\emptyset\}$, %
and is strongly convex with modulus 1 with respect to 
$\|\cdot\|$:
   \[ 
    \forall \theta', \theta''\in \Theta^\circ:~ \langle \partial\omega( \theta')-\partial\omega(\theta''),\  \theta'- {\theta}''\rangle \geq\| \theta'- {\theta}''\|^2. 
    \] 
\item \emph{Bregman distance}: $V_{\theta}(\theta'):=\omega( \theta')-\omega( \theta)-\langle \partial\omega( \theta), \theta'- \theta\rangle$ for all $ \theta\in \Theta^\circ$ and $ \theta'\in \Theta$. 

Note $V_ \theta( \theta') \geq \frac{1}{2} \| \theta- \theta'\|^2 \geq 0$ for all $ \theta\in\Theta^\circ$ and $ \theta'\in\Theta$ follows from the strong convexity of $\omega$. 

\item  \emph{Prox-mapping}: Given a \emph{prox center} $ \theta\in \Theta^\circ$,
  \[
  \Prox_\theta(\xi):=\argmin\limits_{\theta' \in \Theta}\left\{\langle \xi, \theta'\rangle + V_ \theta( \theta')\right\}: \E\to \Theta^\circ.
  \]

When the {\dgf} is taken as the squared $\ell_2$-norm, the prox mapping becomes the usual projection operation of the vector $
\theta-\xi$ onto $\Theta$.
\item \emph{$\omega$-center}: $\theta_\omega:=\argmin\limits_{\theta\in \Theta}\omega(\theta)$.
\item \emph{Set width}: $\Omega:=\max\limits_{\theta\in \Theta}V_{\theta_\omega}(\theta)\leq\max\limits_{\theta\in \Theta}\omega(\theta)-\min\limits_{\theta\in \Theta}\omega(\theta)$.
\end{itemize}
When functions $\{\ell_t^{sim}(\theta)\}_{t \in [T]}$ are convex in $\theta$, online MD as stated in \cite[Algorithm 1]{ho2019exploiting} is applicable to guarantee a sublinear regret bound on $R_T(\{\ell^{sim}_t\}_{t \in [T]}, \{\theta_t\}_{t \in [T]})$, which further bounds regrets with respect to the other loss functions as discussed in Section~\ref{sec:pi}. 

\begin{theorem} \label{thm:pi}
\cite[Theorem 1]{ho2019exploiting}
Suppose $\Theta$ is convex and $\ell_t: \Theta \mapsto \R$ is a convex function for $t \in [T]$. Suppose there exists $G \in (0,\infty)$ such that all the subgradients $s_t$ of $\ell_t$ are bounded, i.e., $\max_{s_t\in\partial \ell_t(\theta)}\norm{s_t}_* \leq G$ for all $\theta \in \Theta$ and $t \in [T]$. Let the step size $\eta_t$ be chosen as $\eta_t = \frac{2\Omega}{G^2 T}$. At time step $t$, using the online Mirror Descent algorithm, we generate $\theta_{t+1}$ as
\begin{equation} \label{eq:thetaupdate}
\theta_{t+1} \coloneqq \Prox_{\theta_t}(\eta_t s_t) = \argmin_{\theta \in \Theta} \left\{ \left\la \eta_t s_t, \theta \right\ra + V_{\theta_t}(\theta) \right\}, 
\end{equation}
where $s_t\in\partial \ell_t(\theta_t)$. 
Then the sequence $\{\theta_t\}_{t \in [T]}$ satisfies
$R_T(\{\ell_t\}_{t \in [T]}, \{\theta_t\}_{t \in [T]}) \leq \sqrt{2\Omega G^2 T}$.
\end{theorem}
In applying Theorem~\ref{thm:pi} to the loss functions $\{\ell^{sim}_t\}_{t \in [T]}$, we have the subgradient $s_t = c(y_t) - c(x(\theta_t;u_t))$. So, it suffices to set $G \geq 2 \max \{\norm{c(x)}_*: x \in \mathcal{X}\}$. The set width $\Omega$ depends on $\Theta$ only and can be computed for a given $\Theta$ and Bregman distance explicitly.

\subsection{Implicit Online Learning with a Solution Oracle}\label{sec:implicitOL}
We next review the implicit OL with a solution oracle from \cite{DongCZ2018}. This algorithm was first introduced  in its general form in \cite{Kulis2010}.

The \textit{implicit online learning} algorithm computes
\begin{equation} \label{implicitupdate}
\theta_{t+1} \coloneqq \text{argmin}_{\theta \in \Theta}L_t(\theta),
\end{equation}
where $L_t(\theta) = V_{\theta_t}(\theta) + \eta_t \ell_t(\theta)$ and $V_{\theta}(\theta')$ is the Bregman distance, $\eta_t$ is a step size. This approach does not rely on the first-order oracle on $\ell_t$, but rather assumes the existence of a \emph{solution oracle} to solve \eqref{implicitupdate}. \cite{Kulis2010} establish the following regret bound on the OL using implicit update \eqref{implicitupdate}.

\begin{theorem}[{\cite[Theorem 3.2.]{Kulis2010}}] \label{IOLregret}
Suppose $\Theta$ is convex, and $\ell_t: \Theta \mapsto \R$ is a convex and differentiable function for $t \in [T]$. Let $\theta^*$ be the offline optimal solution to $\min_{\theta \in \Theta} \sum_{t \in [T]} \ell_t(\theta)$. For any $0 < \alpha_t \leq \frac{L_t(\theta_{t+1})}{L_t(\theta_t)}$ for $t \in [T]$, for any step size $\eta_t>0$, an implicit OL algorithm with the update rule \eqref{implicitupdate} attains
\begin{equation} \label{Kulisregretbd}
R_T(\{\ell_t\}_{t \in [T]}, \{\theta_t\}_{t \in [T]}) \leq \sum_{t \in [T]} \frac{1}{\eta_t} \left [(1-\alpha_t)\eta_t \ell_t(\theta_t) + V_{\theta_t}(\theta^*)-V_{\theta_{t+1}}(\theta^*) \right].
\end{equation}
\end{theorem}

When $\ell_t$ is a convex and Lipschitz continuous function of $\theta$ and the domain $\Theta$ has a finite width with respect to the selected Bregman divergence, the regret bound \eqref{Kulisregretbd} further results in a $O(\sqrt{T})$ bound on $R_T(\{\ell_t\}_{t \in [T]}, \{\theta_t\}_{t \in [T]})$. 
\begin{theorem} \label{thm:IOLRegretBdExplicit}
Suppose $\Theta$ is convex, and for each $t \in [T]$, $\ell_t: \Theta \mapsto \R$ is a convex function of $\theta$ that is uniformly Lipschitz continuous with parameter $G$, and suppose $\max_{\theta_1, \theta_2 \in \Theta}V_{\theta_1}(\theta_2) \leq \widehat{\Omega}$. Then, by choosing $\eta_t = \frac{\sqrt{\widehat{\Omega}}}{G} \frac{1}{\sqrt{t}}$ for $t \in [T]$, an implicit OL algorithm with the update rule \eqref{implicitupdate} attains
\begin{equation} \label{eq:thm:regretbdexplicit}
R_T(\{\ell_t\}_{t \in [T]}, \{\theta_t\}_{t \in [T]}) \leq 2\sqrt{\widehat{\Omega}G^2 T}.  
\end{equation}
\end{theorem}
To apply Theorem \ref{thm:IOLRegretBdExplicit}, we can choose the Lipschitz parameter $G$ by definition. For instance, with the loss functions $\{\ell_t^{sim}\}_{t \in [T]}$, we have $\lvert \ell_t^{sim}(\theta) - \ell_t^{sim}(\theta^{'}) \rvert = \lvert \langle \theta - \theta^{'},c(y_t)-c(x(\theta_t;u_t)) \rangle \rvert \leq \lVert \theta - \theta^{'} \rVert \lVert c(y_t)-c(x(\theta_t;u_t)) \rVert$, hence $G \geq 2 \max \{\norm{c(x)}_*: x \in \mathcal{X}\}$ suffices. 
Alternatively, with $\{\ell_t^{pre}\}_{t \in [T]}$, $\lvert \ell_t^{pre}(\theta) - \ell_t^{pre}(\theta^{'}) \rvert = \lvert \langle x(\theta^{'};u_t) - x(\theta;u_t), 2y_t - x(\theta^{'};u_t) - x(\theta;u_t) \rangle \rvert \leq \lVert x(\theta^{'};u_t) - x(\theta;u_t)\rVert \lVert 2y_t - x(\theta^{'};u_t) - x(\theta;u_t) \rVert$, so we need additional information about $x(\theta;u_t)$ to decide a suitable $G$. The set width $\hat{\Omega}$ only depends on $\Theta$ and the Bregman distance definition.

\subsection{Comparison with the Existing Approaches} \label{sec:litcomp}
\cite{BarmannMPS2018} study online inverse optimization under perfect information where the agent's objective $f$ is a bilinear function of $\theta$ and $x$, i.e., $f({x};\theta)=\la \theta, {x}\ra$. 
They suggest using the online gradient descent and the Multiplicative Weights Update (MWU) algorithms to generate $\{\theta_t\}_{t \in [T]}$ estimates and show via separate analysis that the resulting estimates have vanishing average losses with respect to $\ell^{est}$ and $\ell^{sub}$ (at the rate $O({1/\sqrt{T}})$) but do not present their regret bounds or analyze $\ell^{pre}$ loss. 
Note that both online gradient descent and MWU algorithm are simply special cases of the online MD algorithm customized to the geometry of the problem domain. 
Moreover, the setting studied in \cite{BarmannMPS2018} clearly satisfies our Assumption~\ref{assp:fstr} and the perfect information assumption, hence we can utilize our $\ell^{sim}$-based OL framework equipped with online MD and directly derive average regret bounds of $\textstyle O({1/\sqrt{T}})$ on $\ell^{sim}$, $\ell^{sub}$ and $\ell^{est}$. 
In addition, as opposed to the simple bilinear form of $f$ considered in \cite{BarmannMPS2018}, our framework can handle more general functions $f$ in the forward problem when $f_1({x};{u})$ and/or $f_2(\theta; {u})$ are nontrivial. In this respect, the case of strongly convex $f_1({x}; {u})$ is of special interest, since in this case, through Corollary~\ref{cor:strconv}, our framework also leads to regret bound with respect to $\ell^{pre}$. 

\cite{DongCZ2018} study the following problem where $f$ is linear in $\theta$ and strongly convex in $x$
\begin{gather} \label{sys:exli}
    \begin{aligned}
    \min_{{x}} \left\{\textstyle  \frac{1}{2} {x}^\top P {x} - \langle \theta_{true}, {x} \rangle:~ {x} \in \mathcal{X}({u}) \right\} .
    \end{aligned}
\end{gather}
Here, $P$ is a positive definite matrix and $\mathcal{X}({u})$ is the agent's feasible domain determined by the external signal fixed as ${u}$. In this setting, \cite{DongCZ2018} propose a regret minimization algorithm utilizing the implicit OL method (\cite{Kulis2010}) with a nonconvex MISOCP oracle. They focus on the prediction loss $\ell^{pre}$, and establish a $O(\sqrt{T})$ bound on $R_T(\{\ell_t^{pre}\}_{t \in [T]}, \{\theta_t\}_{t \in [T]})$ whenever $\ell_t^{pre}(\theta)$ is a convex function of $\theta$. A main limitation of their approach is that the convexity of $\ell^{pre}$ does not hold in general. 
Although they identify a technical sufficient condition \cite[Assumption 3.3]{DongCZ2018} that can guarantee convexity of $\ell^{pre}$, they also remark that this condition is restrictive and very hard to verify in practice even for the simplest form of problem classes. In fact, the only example they identify as satisfying their assumption is when the agent's optimization problem is \eqref{sys:exli} and the set $\CX({u})$ must \emph{always} contain the minimizer of the unrestricted objective minimization problem, i.e., $P^{-1}\theta_{true} \in \CX({u})$ for all possible ${u}$. 

When the agent's problem has the specific form of \eqref{sys:exli}, the algorithm from \cite{DongCZ2018} updates $\theta_{t+1}$ as the optimal solution of the following bilevel program:
\[
\theta_{t+1} \coloneqq \argmin_{\theta \in \Theta} \left\{\frac{1}{2}\norm{\theta-\theta_t}^2 + \eta_t \norm{{y}_t - {x}(\theta;{u}_t)}^2: {x}(\theta;{u}_t) \in \argmin_{{x}} \left\{ 
\frac{1}{2} {x}^\top P {x} -\la \theta,{x} \ra: {x} \in \CX({u}_t)
\right\}
\right\}.
\]
It was shown in \cite{DongCZ2018} that when the feasible domain $\CX({u}_t)$ is polyhedral, this bilevel program can be represented as a MISOCP. Consequently, the implicit OL algorithm of \cite{DongCZ2018} utilizes an MISOCP based solution oracle to generate $\{\theta_t\}_{t \in [T]}$. The main convergence result \cite[Theorem 3.2]{DongCZ2018} proves that under their assumptions by choosing the step size $\eta_t \propto 1/\sqrt{t}$, the sequence of estimates$\{\theta_t\}_{t \in [T]}$ generated with the above update yields a $O(\sqrt{T})$ bound on the regret $R_T(\{\ell_t^{pre}\}_{t \in [T]}, \{\theta_t\}_{t \in [T]})$.

Note that the format of $f$ in \eqref{sys:exli} satisfies our Assumption~\ref{assp:fstr}, and consequently $\ell_t^{sim}$ is guaranteed to be convex for any $\CX({u})$. Therefore, our OCO framework based on minimizing regret for loss functions $\{\ell_t^{sim}\}_{t \in [T]}$ is applicable to \eqref{sys:exli}.  
In addition, in the perfect information setting, through Proposition~\ref{prop:pi} and Corollary~\ref{cor:strconv}, our framework can provide regret bounds with respect to \emph{all} of $\ell^{sim}$, $\ell^{ag}$, $\ell^{est}$, $\ell^{sub}$ and $\ell^{pre}$, without further structural assumptions on the agent's domain. 
In contrast, the implicit OL approach of \cite{DongCZ2018} for minimizing regret with respect to $\ell^{pre}$ requires additional conditions on the agent's domain (see \cite[Assumptions 3.1,3.2,3.3]{DongCZ2018}) in order to guarantee a regret bound. In particular, it is specifically focused on $\ell^{pre}$ and provides no insight on other performance measures of interest captured by $\ell^{est}$ and $\ell^{sub}$ either. 
Moreover, online MD in our framework uses a much simpler (and computationally faster) first-order oracle in contrast to the expensive MISOCP oracle in  the implicit OL approach of \cite{DongCZ2018}. 
One aspect that \cite{DongCZ2018} emphasize but we do not address is the noise in observations: they prove theoretical regret bounds for the case where ${y}_t$ is a noisy observation of ${x}(\theta_{true};{u}_t)$. However, as previously discussed, theoretical guarantees of our framework based on $\ell^{sim}$ fail to readily extend to the imperfect information setup.

\section{Computational Study} \label{sec:app}
We perform numerical experiments on a practical application that is motivated by a company (learner) seeking to learn about its customer's (agent's) preferences in a changing market. We assume the customer is a rational decision maker, and in any given market situation, her/his action reflects accurately her/his optimal preferences. 
These experiments do not aim to provide structural insights on specific instances, rather, our main purpose is to demonstrate the performance of $\ell^{sim}$ based OCO algorithms from various aspects and the comparison with an alternative $\ell^{pre}$-based approach in \cite{DongCZ2018}.

We first focus on the case when perfect information is available, i.e., there is no noise in learner’s observations of the agent’s optimal actions, and address three main questions. First, are there notable performance differences among OL algorithms based on different oracles? Second, how do the algorithm performances vary in terms of different loss functions? Third, does the structure of the agent’s feasible region affect complexity of the learning problem and the algorithm performances? While discussing these questions, we also compare against existing algorithms from the literature.

In the second part of our numerical study, we examine the robustness of these OL algorithms under imperfect information, i.e., when there is random noise to the learner’s observations of the agent’s optimal actions. Recall that in the imperfect information setup, our OL based approach is not guaranteed to provide low regret guarantees; see Appendix~\ref{app:imperinfo} for a discussion of the main issues. Hence, these experiments essentially shed light to their empirical performance in the noisy setup.

All algorithms are coded in Python 3.8, and Gurobi 8.1.1 with default settings is used to solve the mathematical programs needed for the subproblems associated with the corresponding oracles. We limit the solution time of each mathematical program to be at most 3600 seconds. We have not hit this imposed time limit in any of our experiments. All experiments are conducted on a server with 2.8 GHz processor and 64GB memory.

\subsection{Problem Instances} \label{app:instance}
We consider a market with $n$ products that evolves over a finite time horizon $T$, e.g., the product prices change. These changes consequently impact the agent's feasible actions; in this case, agents are customers interested in purchasing the products. For each $t \in [T]$, we let $u_t$ denote the market parameters relevant to the agent's decisions at period $t$. When constraint parameters are fixed as $u_t$, an agent's action ${x}(\theta_{true}; u_t)$ is an optimal solution to an optimization problem parametrized by $u_t$ and $\theta_{true}$, where $\theta_{true}$ captures the agent's preferences over the products. 
We model the agent's optimization problem as a maximization of her/his utility function subject to feasibility constraints. 
The learner knows the agent's decision problem up to the parameter vector $\theta_{true}$, and the learner's goal is to estimate $\theta_{true}$ using observations of the agent's actions ${y}_t$ in response to the market conditions $u_t$ at each period $t\in[T]$.

We study two different forms for the agent's utility function.
\begin{enumerate}
\item[(a)] For direct comparison with \cite{DongCZ2018}, we examine the case where the agent's utility function has the quadratic form~\eqref{sys:exli}, i.e., the agent's action ${x}(\theta_{true}; u_t)$ is given by
\begin{gather}
    \begin{aligned}
   {x}(\theta_{true}; u_t) \coloneqq
   \argmax_{{x}} \left\{ -\frac{1}{2}{x}^\top P{x} + \langle \theta_{true}, {x} \rangle 
    :~ {x} \in \mathcal{X}(u_t) \right\},
    \end{aligned} \label{ex:prefinmkt}
\end{gather}
where $P \in \Se_{++}^n$ is a fixed positive definite matrix known by both the learner and the agent and $\mathcal{X}(u_t)$ represents the domain for the agent's feasible actions determined by the market parameters $u_t$. 

\item[(b)] We also examine a second setup where the agent has a CES utility function with $\rho = 2$. Hence, in period $t$, the agent's action ${x}(\theta_{true}; u_t)$ is given by
\begin{gather}
    \begin{aligned}
   {x}(\theta_{true}; u_t) \coloneqq
   \argmax_{{x}} \left \{\sum_{i \in [n]} -(\theta_{true})_i x_i^2
    :~ {x} \in \mathcal{X}(u_t) \right\}.
    \end{aligned} \label{ex:prefces}
\end{gather}
Note that this setup with a CES utility has not been previously studied in an OL framework.
\end{enumerate}

These particular forms of utility functions in \eqref{ex:prefinmkt} and \eqref{ex:prefces} imply that the dimensions of $\theta$ and $x$ are the same, i.e., $p=n$. Moreover, observe that both of the objective functions in~\eqref{ex:prefinmkt}~and~\eqref{ex:prefces} satisfy Assumption~\ref{assp:fstr}, and thus in both cases  $\ell^{sim}_t(\theta)$ is convex in $\theta$.

To identify the impact of agent's feasible region on the complexity of the problem as well as on the performance of the learning algorithms, we experiment on a variety of settings for $\mathcal{X}(u_t)$. 
\begin{enumerate}
    \item[(i)] \emph{Continuous knapsack} domain: in this setting, we impose only a budget constraint on the agent: $\mathcal{X}(u_t) = \mathcal{X}^{ck}({p}_t,b_t) \coloneqq \{{x} \in \R_+^n: \langle {p}_t, {x} \rangle \leq b_t \}$, where the parameters ${p}_t\in\R^n_+$ correspond to the product prices and $b_t\in\R_+$ is the budget available to the customer during time period $t$. Note that both ${p}_t$ and $b_t$ can vary in each time period $t\in[T]$. 
    \item[(ii)] \emph{Continuous polytope} domain: here, we generalize the continuous knapsack domain and model general resource constraints resulting in a polytope as the feasible region $\mathcal{X}(u_t) = \mathcal{X}^{cp}(A_t,{c}_t) = \{{x} \in \R_+^n: A_t {x} \leq {c}_t \}$, where all the parameters are nonnegative. 
    \item[(iii)] \emph{Binary knapsack} domain: in this case, we again impose a single budget constraint, but also require that the agent's action is a binary vector: $\mathcal{X}(u_t) = \mathcal{X}^{bk}({p}_t,b_t) \coloneqq \{{x} \in \{0,1\}^n: \langle {p}_t, {x} \rangle \leq b_t \}$.
    \item[(iv)] \emph{Equality constrained knapsack} domain: that is, $\mathcal{X}(u_t) = \mathcal{X}^{eck}({p}_t,b_t) \coloneqq \{{x} \in \R_+^n: \langle {p}_t, {x} \rangle = b_t \}$. 
\end{enumerate}
We ran experiments with the utility function \eqref{ex:prefinmkt} where we choose the matrix $P$ to be a positive definite diagonal matrix and generate each of its diagonal entries $P_{ii}$ by first drawing a number from $[1,21]$ uniformly and then normalizing the drawn vector $(P_{11},\ldots,P_{nn})$ to have a unit $\ell_1$-norm, and we also set the domain to be $\mathcal{X}^{ck}(p_t,b_t)$, $\mathcal{X}^{cp}(A_t,c_t)$, or $\mathcal{X}^{bk}(p_t,b_t)$. 
In the case of CES utility function \eqref{ex:prefces}, for implementation simplicity, we use instances with the domain $\mathcal{X}^{eck}(u_t)$.

In all of our experiments, we consider a market with $n = 50$ goods. We compare OL algorithms by running $T = 500$ iterations on a batch of $50$ randomly generated instances for each setting. The domain $\Theta$ is set be a unit simplex, i.e., $\Theta = \left\{\theta \in \mathbb{R}_+^n: \sum_{i \in [n]} \theta_i = 1\right\}$. We follow the same instance generation methodology used in \cite[Section 4.1]{BarmannMPS2018} for generating the true parameter $\theta_{true}$ and the agent's domain $\mathcal{X}(u_t)$. In each instance, $\theta_{true}$ is obtained by drawing a random sample from a uniform distribution over $[1,1000]^n$ and then normalizing the sampled vector to have a unit $\ell_1$-norm. 
In the case of $\mathcal{X}^{ck}({p}_t,b_t),~\mathcal{X}^{bk}({p}_t,b_t)$, and $\mathcal{X}^{eck}({p}_t,b_t)$, for all $t \in [T]$, the constraint parameters ${p}_t,b_t$ are generated randomly as follows: ${p}_t$ is set as $\theta_{true} + 100\cdot \mathbf{1}_n + {r}$, where ${r}$ is an integer vector sampled from a \emph{discrete uniform} distribution over the collection of integer vectors in $[-10,10]^n$ (\emph{numpy.random.randint} function is used). The budget $b_t$ is selected uniformly random from the range $\left[1,\sum_{i=1}^n (p_t)_i\right]$.  In the case of continuous polytope domain $\mathcal{X}^{cp}(A_t,{c}_t)$, we choose $A_t$ as an $m \times n$ matrix with $m = 10$, where each row of $A_t$ is generated in the same way as ${p}_t$, and  each coordinate $i$ in the vector ${c}_t$ is drawn uniformly random from $[1,\sum_{j=1}^m (A_t)_{ji}]$. 

In the OL setup, at time step $t$, the learner observes the signal $u_t$ and the agent's action, and uses the information revealed so far to construct the estimate $\theta_{t+1}$. Under perfect information, we have ${y}_t={x}(\theta_{true}; u_t)$ for all $t$; under imperfect information, we assume ${y}_t={x}(\theta_{true}; u_t) + \epsilon_t$, where $\epsilon_t$ denotes the random noise. 
In the noisy setup, each coordinate in $\epsilon_t$ is randomly drawn from a uniform distribution over a given range. We consider two ranges to simulate small and large noises, and we choose the range bounds based on the average $\delta := \frac{1}{T}\sum_{t \in [T]}\norm{{x}(\theta_{true};{u}_t)}$: the small noises are generated with the range $[-\delta/n, \delta/n]$, and the large noises are generated with $[-\delta, \delta]$.

\subsection{Implementation Details} \label{app:implement}

In order to compute the estimates $\{\theta_t\}_{t \in [T]}$, we implement three OL algorithms and compare their performances. By taking advantage of the convexity of $\ell^{sim}$, we design two OL algorithms minimizing regret with respect to $\ell^{sim}$: one equipped with a first-order oracle and one with a solution oracle. In addition, for comparison with the literature, we implemented another implicit OL algorithm with a solution oracle aimed to minimize the regret associated with $\ell^{pre}$, i.e., the one from \cite{DongCZ2018} that utilizes an MISOCP solution oracle. We provided precisely the same dynamic observations, i.e., the realizations of signals $u_t$ along with the agent's optimum action ${x}(\theta_{true}; u_t)$ in each iteration $t\in[T]$ to all of these OL algorithms.

In the case of OL with the first-order oracle, because $\Theta$ is a unit simplex, we use the negative entropy function $\omega(\theta) = \sum_{i =1}^n \theta_i \ln(\theta_i)$ as the distance generating function in the definition of Bregman distance $V_{\theta_t}(\theta)$. Then, the update rule \eqref{eq:thetaupdate} for the OL with first-order oracle is given explicitly by the following formula, where $s_t(\theta_t)$ is a subgradient of $\ell^{sim}_t(\theta)$ at $\theta_t$.
\[
(\theta_{t+1})_i = \frac{(\theta_t)_i \exp(-\eta_t (s_t(\theta_t))_i)}{\sum_{j=1}^n (\theta_t)_j \exp(-\eta_t (s_t(\theta_t))_j)}, ~\text{for all } i \in [n].
\]

The main challenge in the implementation of implicit OL algorithm with a solution oracle is whether one can design a computationally tractable solution oracle. 
When the loss function $\ell_t(\theta)$ used in the implicit OL involves $x(\theta;u_t)$, as is the case in all loss functions from Section \ref{sec:inverseprob} except $\ell_t^{sim}(\theta)$, \eqref{implicitupdate} is a bilevel program. Bilevel programs are difficult to solve in general, but can be reformulated into a single level problem using KKT conditions of the inner level problem whenever the inner level is a convex problem. In contrast to this, when $\ell^{sim}_t$ is used as the loss function in an implicit OL algorithm, \eqref{implicitupdate} becomes a single level optimization problem in $\theta$ and thus the solution oracle becomes much simpler. Consequently, 
we study two variants of the implicit OL algorithm based on $\ell^{sim}$ and $\ell^{pre}$ that are necessarily equipped with different solution oracles.

In the first variant, we design an implicit OL algorithm to minimize the regret with respect to the loss function $\ell^{sim}$. Using the squared Euclidean norm as the d.g.f., we arrive at the implicit OL algorithm with a solution oracle that updates $\theta_{t+1}$ as the optimal solution to 
\begin{gather*} 
    \begin{aligned}
    \theta_{t+1} \coloneqq \argmin_{\theta \in \Theta} \frac{1}{2}\norm{\theta-\theta_t}^2 + \eta_t \ell^{sim}_t(\theta).
    \end{aligned}
\end{gather*}
Under Assumption~\ref{assp:fstr}, $\ell^{sim}_t(\theta)$ is convex in $\theta$, and when the domain $\Theta$ is convex, the above problem is a convex program. Therefore, the implementation requires only a convex solution oracle; see Appendix~\ref{sec:app-sol} for the explicit formulations of these oracles. 

For comparison purposes, we implement a second variant of the implicit OL algorithm minimizing regret with respect to the loss function $\ell^{pre}$. By following the same approach taken in \cite{DongCZ2018}, we use the squared Euclidean norm as the d.g.f., and the resulting solution oracle updates $\theta_{t+1}$ with the following bilevel program, where the inner level computes $x(\theta,u_t)$ used in $\ell^{pre}_t(\theta)$:
\begin{gather*}
    \begin{aligned}
    \theta_{t+1} \coloneqq \argmin_{\theta \in \Theta} \frac{1}{2}\norm{\theta-\theta_t}^2 + \eta_t \ell^{pre}_t(\theta).
    \end{aligned}
\end{gather*}
When the agent is maximizing a concave objective function over a polyhedral domain $\mathcal{X}(u_t)$, we can reformulate the above bilevel program into a mixed integer program (MIP).

Consequently, at time $t$, this $\ell^{pre}$-based implicit OL algorithm requires a nonconvex oracle given by the MIP formulation to obtain $\theta_{t+1}$. In the case of \eqref{ex:prefinmkt}, it was demonstrated in \cite{DongCZ2018} that when the domain $\mathcal{X}(u_t)$ is polyhedral, the MIP reformulation admits a nice MISOCP structure due to the quadratic objective. For completeness, we provide the MISOCP reformulation of this solution oracle in Appendix \ref{sec:app-sol}. Note that due to the advanced capabilities of modern MIP solvers, the resulting MISOCP still remains computationally tractable whenever the scale of the agent's problem is relatively small. 

On the other hand, when the domain of the inner problem $\mathcal{X}(u_t)$ is nonconvex, e.g., when we consider $\mathcal{X}^{bk}(p_t,b_t)$ that involves binary variables, or when the agent maximizes a nonconcave function over a convex domain $\mathcal{X}(u_t)$ as in the case of \eqref{ex:prefces} for $\theta\notin\R^n_+$, 
we no longer have access to KKT based optimality certificates for the inner problem. Consequently, in such cases, we do not know how to design a computationally tractable solution oracle, and this is an open question. Therefore, we did not experiment with the $\ell^{pre}$-based implicit OL algorithm in these cases. 

\subsection{Perfect Information Experiments} \label{app:perexp}
In this section, we discuss our numerical results along with plots that highlight our key observations pertinent to the questions of interest to the perfect information case listed at the beginning of Section \ref{sec:app}.

\subsubsection{Learning a Quadratic Utility Function} \label{app:per-qua}
In this case, we assume that the agent's utility function is of form \eqref{ex:prefinmkt}. We first compare the performance of the three OL approaches in terms of both average regret performance and the solution time. Figures \ref{fig:50-n50K-reg} and \ref{fig:50-n50P-reg} display the means of average (expected) regret performance of the iterates $\{\theta_t\}_{t\in[T]}$ returned by the OL algorithms with respect to all five loss functions of interest for the instances where the agent's domain is of continuous knapsack and polytope type, respectively; the shaded areas indicate 95\% confidence interval for the means.
These means are computed based on all $50$ random instances generated in the experiment. 
In Figures \ref{fig:50-n50K-reg} and \ref{fig:50-n50P-reg} (and all the later ones as well), the scale of loss functions naturally differ because the associated regret and loss values are evaluated with respect to different terms present in their corresponding loss definitions. 
In terms of the rate at which the average regret converges, in the case of the continuous knapsack instances, Figure \ref{fig:50-n50K-reg} shows that regardless of the loss function used to evaluate the performance, all three OL algorithms have quite similar performances.
This empirical observation is in line with the theoretical regret guarantees given in Section~\ref{sec:OL}; recall that this particular domain type was the focus of \cite{DongCZ2018}, and their analysis presents some restrictive assumptions guaranteeing convergence of their approach on this type of instances. For the continuous polytope instances, Figure \ref{fig:50-n50P-reg} demonstrates similar performances from the two OL algorithms based on $\ell^{sim}$, but highlights the drastically different performance of the implicit OL with the $\ell^{pre}$-minimizing solution oracle, which now leads to average regrets converging to non-zero values.
Recall from Section \ref{sec:implicitOL}, the regret convergence of an implicit OL algorithm with a solution oracle requires the convexity of the selected loss function. In fact, \cite{DongCZ2018} adopt further strong assumptions on $\mathcal{X}(u_t)$ to guarantee that $\ell^{pre}$ is a convex function of $\theta$ when the agent's problem is of form~\eqref{ex:prefinmkt} with $\mathcal{X}(u_t) = \mathcal{X}^{ck}({p}_t,b_t)$. Our empirical results indicate that these assumptions are indeed hard to satisfy in general and our randomly generated continuous polytope instances do not necessarily satisfy their required assumption. In contrast, since $\ell^{sim}$ is guaranteed to be a convex function of $\theta$ when the agent's problem is of form~\eqref{ex:prefinmkt} regardless of the structure of the agent's domain $\mathcal{X}(u_t)$, the average regrets of the $\ell^{sim}$-based implicit OL algorithm with the solution oracle converge to zero for instances with polytope domain as well. Furthermore, we note that the regret convergence of the $\ell^{sim}$-based implicit OL algorithm with the solution oracle is slightly better than the OL with the first-order oracle in both types of instances.

\begin{figure}
\centering
  \includegraphics[width=0.9\linewidth]{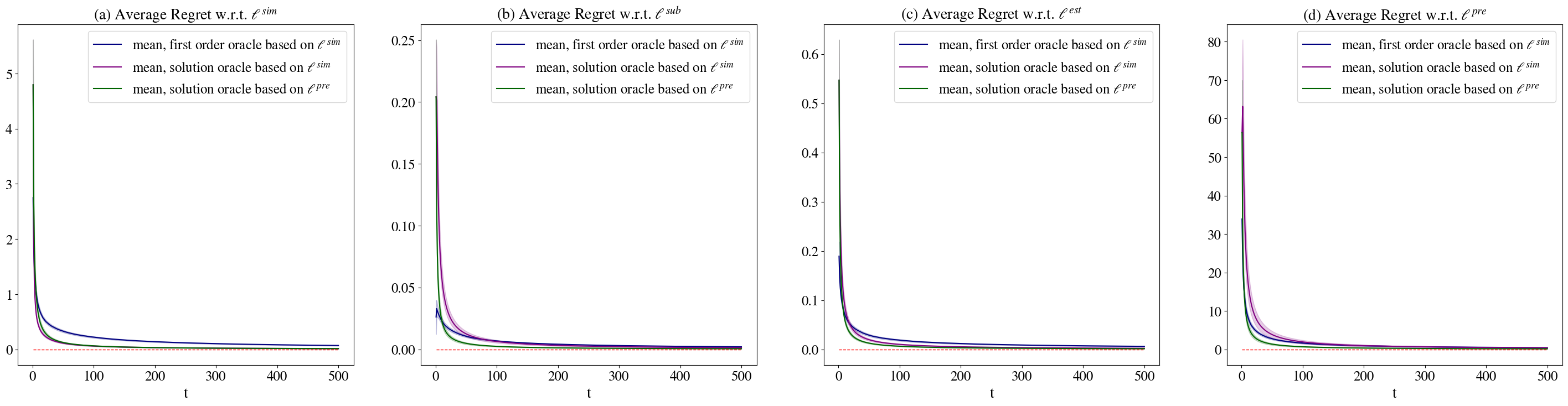}
  \caption{Means of average regret with respect to different loss functions over $T=500$ iterations for continuous knapsack instances; the shaded region is 95\% confidence interval for the means.}
  \label{fig:50-n50K-reg}
\end{figure}

\begin{figure}
\centering
    \includegraphics[width=0.9\linewidth]{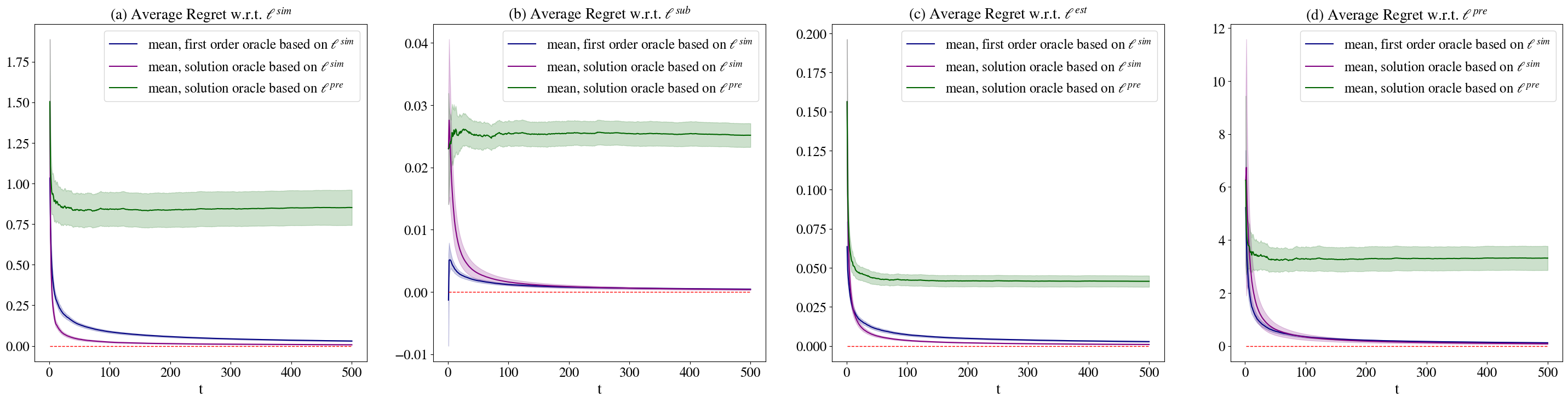}
   \caption{ Means of average regrets with respect to different loss functions over $T=500$ iterations for continuous polytope instances; the shaded region is 95\% confidence interval for the means.}
     \label{fig:50-n50P-reg}
\end{figure}

In our numerical study, we observe almost no variation in terms of the solution time of the OL algorithms across different random instances generated from the same setting.  We thereby report the time spent by all three OL algorithms on a randomly selected instance from our problem set. When computing the solution time at iteration $t$, we always ignore the time taken to find ${x}(\theta_{true};u_t)$. In iteration $t$ of the OL with the first-order oracle, we account for the time to compute ${x}(\theta_t;u_t)$ and generate $\theta_{t+1}$ using the first-order oracle. Lastly, in each iteration of both of the $\ell^{sim}$- and $\ell^{pre}$-based implicit OL algorithms with a solution oracle, we account for the time used by the corresponding solution oracles in updating $\theta_{t+1}$. 
For an arbitrary instance with the continuous knapsack domain, OL with the first-order oracle finishes in about 0.08 seconds, $\ell^{sim}$-based implicit OL with the solution oracle takes 2.03 seconds, and $\ell^{pre}$-based implicit OL with the solution oracle takes 146 seconds. These highlight that, by a significant margin, our OL algorithms minimizing regret with respect to the loss function $\ell^{sim}$ and utilizing the first-order oracle and the solution oracle are much more computationally efficient than the $\ell^{pre}$-based implicit OL with the MISOCP solution oracle one from \cite{DongCZ2018}.

We next analyze whether the agent's domain structure has any visible effect on the overall regret performance of the OL with the first-order oracle. From Figures \ref{fig:50-n50K-reg} and \ref{fig:50-n50P-reg}, we observe that the superiority of the OL with first-order oracle in terms of the average regret is slightly more obvious in the continuous knapsack setting than in the polytope setting. In Figure \ref{fig:binarycompare}, we compare the means of average regrets for the continuous knapsack instances versus the binary knapsack instances. The regret  performances with respect to the loss functions $\ell^{sub}$ and $\ell^{est}$  seem to vary only slightly when the agent's domain type changes from continuous knapsack to binary knapsack; yet these differences are slightly more noticeable in the case of loss functions $\ell^{pre}$ and $\ell^{sim}$. 

\begin{figure}
\centering
 \includegraphics[width=0.9\linewidth]{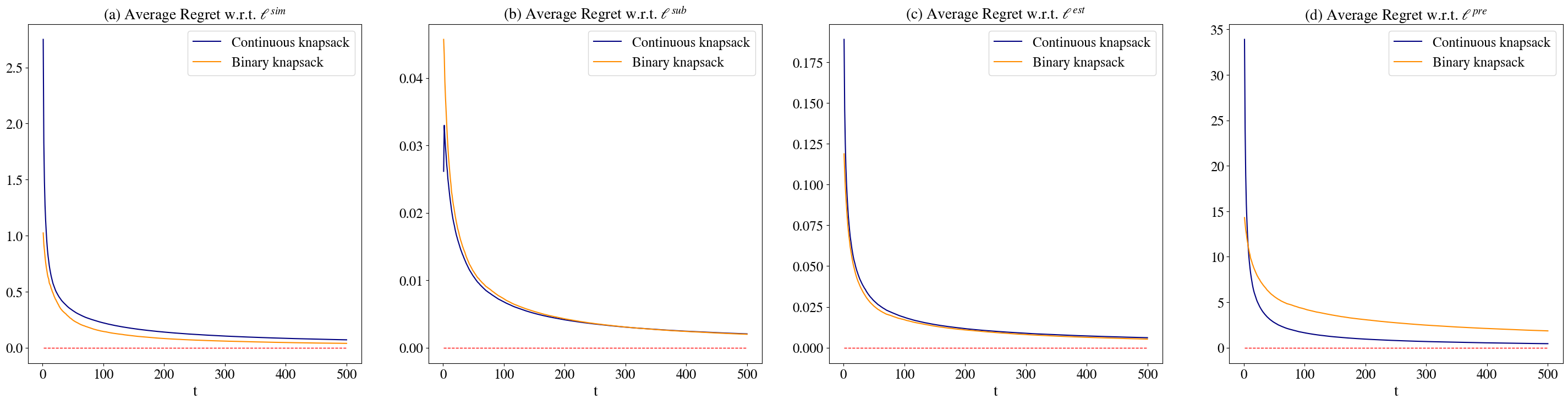}
  \caption{Means of average regret with respect to different loss functions
  contrasting continuous knapsack instances with binary knapsack instances,
  when OL with the first-order oracle is used.}
  \label{fig:binarycompare}
\end{figure}

Lastly, we examine the regret performance of OL with the first-order oracle with respect to different loss functions $\ell(\theta)$. From the experiment results from continuous knapsack and continuous polytope instances (respectively Figures \ref{fig:50-n50K-reg} and \ref{fig:50-n50P-reg}), we observe that the average regret with respect to any of the four loss functions convergences roughly at the same rate, but the corresponding regret bounds differ in their scales. This is not surprising, as the corresponding regrets are based on different terms, e.g., norms of solutions or objective function values, etc. Moreover, recall from Section~\ref{sec:pi} that in the perfect information case the following relationship among the regret bounds with respect to different loss functions (here for simplicity in notation, we denote $R^{sim}_T \coloneqq  R_T(\{\ell^{sim}_t\}_{t \in [T]}, \{\theta_t\}_{t \in [T]})$, etc.) holds: 
$R^{sim}_T \geq  R^{sub}_T +  R^{est}_T \geq \frac{\gamma}{2} R^{pre}_T \geq 0,$
where $\gamma$ is the strong convexity parameter of the quadratic objective function in \eqref{ex:prefinmkt}. Recall that our instance generation guarantees $P \in \Se_{++}^n$, i.e., its smallest eigenvalue $\lambda_{min}(P) > 0$, and then by the definition of strong convexity, we deduce $\gamma = \lambda_{min}(P)$. Figure \ref{fig:losscompare-log} displays (on a logarithm scale) the means of the average regrets for different loss functions for $\theta_t$ estimates generated from the OL with the first-order oracle on instances in which the agent's domain is either a continuous knapsack, polytope, or a binary knapsack type. These results also confirm the theoretical relationship among the regrets for different loss functions we have established in Section~\ref{sec:pi}. 
\begin{figure}
\centering
 \includegraphics[width=0.7\linewidth]{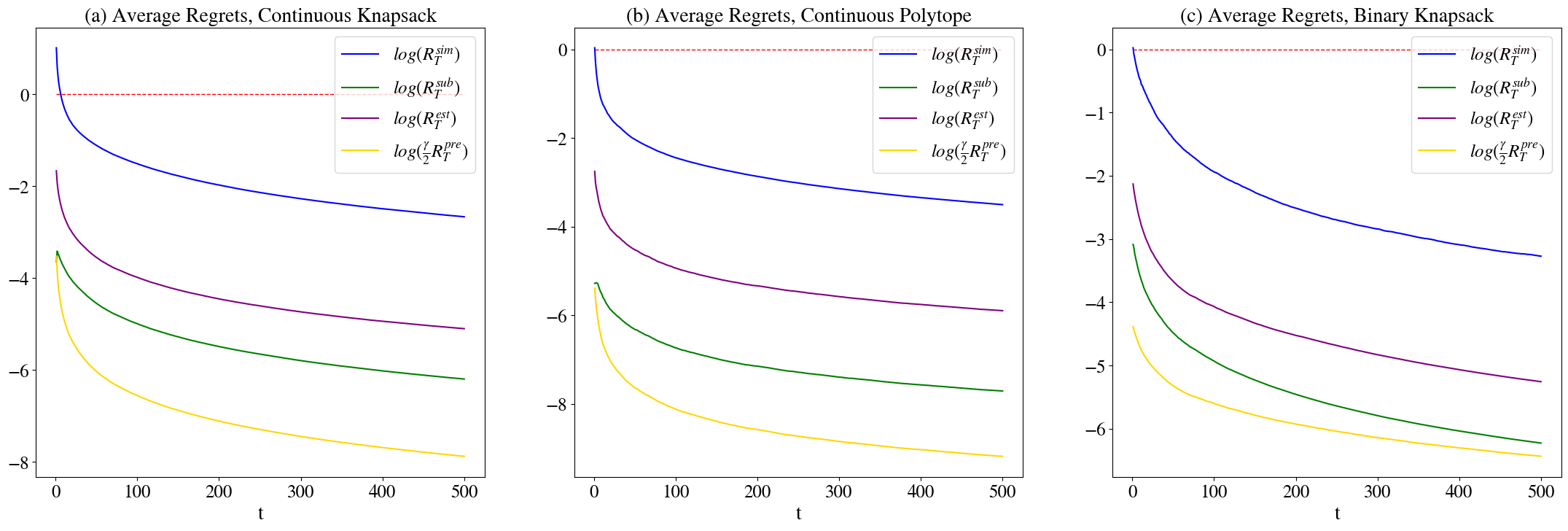}
  \caption{Means of average regret (on a logarithmic scale)  with respect to different loss functions over $T=500$ iterations for (a) continuous knapsack instances, (b) continuous polytope instances, (c) binary knapsack instances, when OL with first-order oracle is used.}
  \label{fig:losscompare-log}
\end{figure}

\subsubsection{Learning a CES Utility Function}\label{sec:app-perces}
Here, we examine the case when the agent's utility function is of form \eqref{ex:prefces} and summarize our findings on the average regrets in Figure \ref{fig:ces-loss}. 
We note that the OL with the first-order oracle has a quite noticeable advantage over the implicit OL with a solution oracle in terms of the regret convergence. In this case, on a typical instance, OL with the first-order oracle takes $0.12$ seconds to complete and $\ell^{sim}$-based implicit OL with the solution oracle takes $2.02$ seconds. 

\begin{figure}
\centering
    \includegraphics[width=0.9\linewidth]{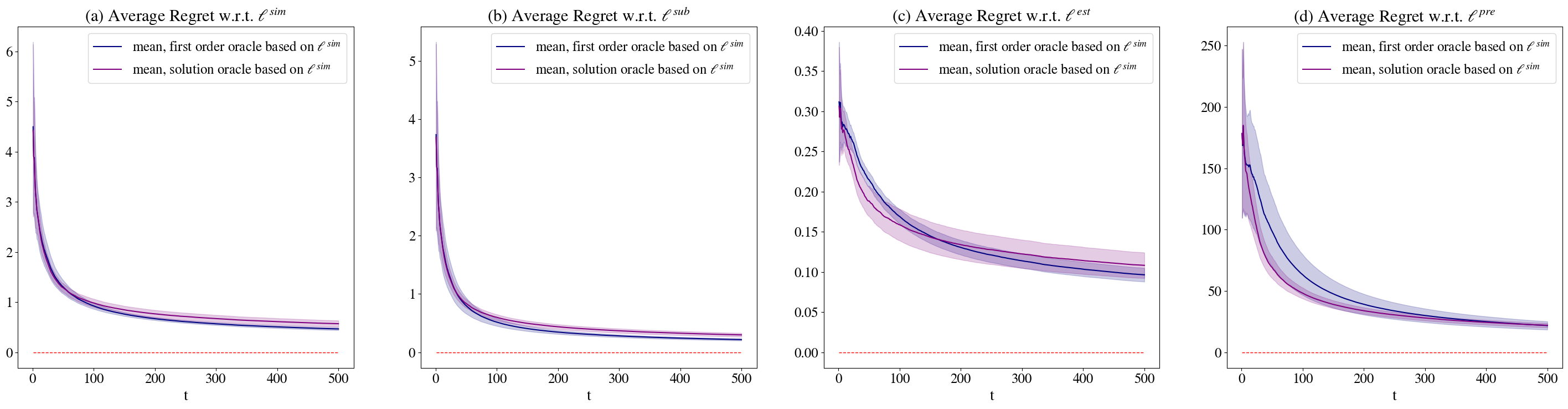}
     \caption{ Means of average regrets with respect to different loss functions over $T=500$ iterations for equality constrained knapsack instances; the shaded region is 95\% confidence interval for the means.}
     \label{fig:ces-loss}
\end{figure}

\subsection{Imperfect Information Experiments} \label{app:imperexp}
We next study the robustness of these OL algorithms when the observations are corrupted with random noise. We test this imperfect information setup on two types of instances where (1) the agent is maximizing a concave quadratic utility function on a continuous knapsack domain, and (2) the agent is maximizing a CES utility function over an equality constrained knapsack domain. We observed that the impact of the noises on the solution time of the OL algorithms was negligible in both of these instance types.

Before discussing the performance and robustness of OL algorithms, we note some key differences with the perfect information case. Because the noisy observations $\{{y}_t\}$ may be suboptimal or even infeasible to the agent's problems, we no longer have the convenient guarantee of Lemma~\ref{lem:minloss} that $\theta_{true}$ gives the offline optimal losses $\min_{\theta \in \Theta} \sum_{t \in [T]} \ell_t(\theta) = 0$  with respect to $\ell^{sub},\ell^{est}$ and $\ell^{pre}$. Specifically, $\ell^{est}_t(\theta) = f({x}({\theta};{u}_t);{\theta}_{true},{u}_t) - f({y}_t;{\theta}_{true},{u}_t) < 0$ is possible when ${y}_t$ is feasible and ${x}({\theta};{u}_t)$ is a better solution than ${y}_t$ for the agent; $\ell^{sub}_t(\theta) = f({y}_t;{\theta},{u}_t) - f({x}({\theta};{u}_t);{\theta},{u}_t) < 0$ can happen when ${y}_t$ is outside the agent's feasible domain $\mathcal{X}({u}_t)$. 
The prediction loss $\ell^{pre}_t(\theta)$ is still nonnegative, but the lowest value is not necessarily $\ell^{pre}_t(\theta_{true})$. Therefore, when the observations $y_t$ are noisy, computing regret $R_T$ requires solving the optimization problem $\min_{\theta \in \Theta} \sum_{t \in [T]} \ell_t(\theta)$, which can be computationally difficult as the term ${x}(\theta;{u}_t)$ makes it a bilevel program. To avoid such difficulty, we plot the outcomes differently: instead of showing average regrets with respect to all of our loss functions, we show only the average regret with respect to $\ell^{sim}$, but we also report the average losses computed with replacing ${y}_t$ by ${x}(\theta_{true};{u}_t)$ in the loss definitions.

\subsubsection{Learning a Quadratic Utility Function} \label{sec:app-imperqua}
We report the results for when the agent's problem has the form~\eqref{ex:prefinmkt} with the domain  $\mathcal{X}(u_t)=\mathcal{X}^{ck}(u_t)$ in Figures \ref{fig:no-knap-50-50}, \ref{fig:no-knap-50-50-per}, \ref{fig:bigno-knap-50-50}, and \ref{fig:bigno-knap-50-50-per}. 
In the imperfect information setup, while only the $\ell^{pre}$ minimization-based implicit OL with the solution oracle has a theoretical guarantee for convergence under further assumptions on $\mathcal{X}(u_t)$, we observe a convergence behavior for both $\ell^{sim}$-based OL algorithms as well. 
For example, in the case of small noises, Figure \ref{fig:no-knap-50-50}-(e) shows that the average regret with respect to $\ell^{sim}$ has a quite similar convergence trend as in Figure \ref{fig:50-n50K-reg}, and Figure \ref{fig:no-knap-50-50-per} demonstrates that all OL algorithms also lead to effective convergence of the average losses measured relative to the true actions. Figures \ref{fig:bigno-knap-50-50}, \ref{fig:bigno-knap-50-50-per} show that the effects of the noises become more noticeable when their magnitudes are larger as the average regret with respect to $\ell^{sim}$ appears to converge much more slowly  for the $\ell^{sim}$-based OL algorithms and the $\ell^{pre}$-based implicit OL with the solution oracle seems to fail to converge, and there are cases of negative average losses with respect to a number of loss functions. It is notable that in the case of $\ell^{sim}$-based OL algorithms, the average losses in terms of $\ell^{pre}$ computed with respect to ${x}(\theta_{true};{u}_t)$ decrease with $T$ even under large noises, which indicates that these OL algorithms' predictions of the agent's actions are becoming more accurate as $T$ increases. We further note that such trends are most explicit in the $\ell^{sim}$-based OL with the first-order oracle. 
These results demonstrate that  $\ell^{sim}$-based OL algorithms have some degree of robustness for certain types of imperfect information, and  the performance of the $\ell^{sim}$-based OL algorithm with first-order feedback seems to be slightly superior in the noisy setup. 

\begin{figure}
\centering
\begin{subfigure}{\textwidth}
\centering
\includegraphics[width=0.7\linewidth]{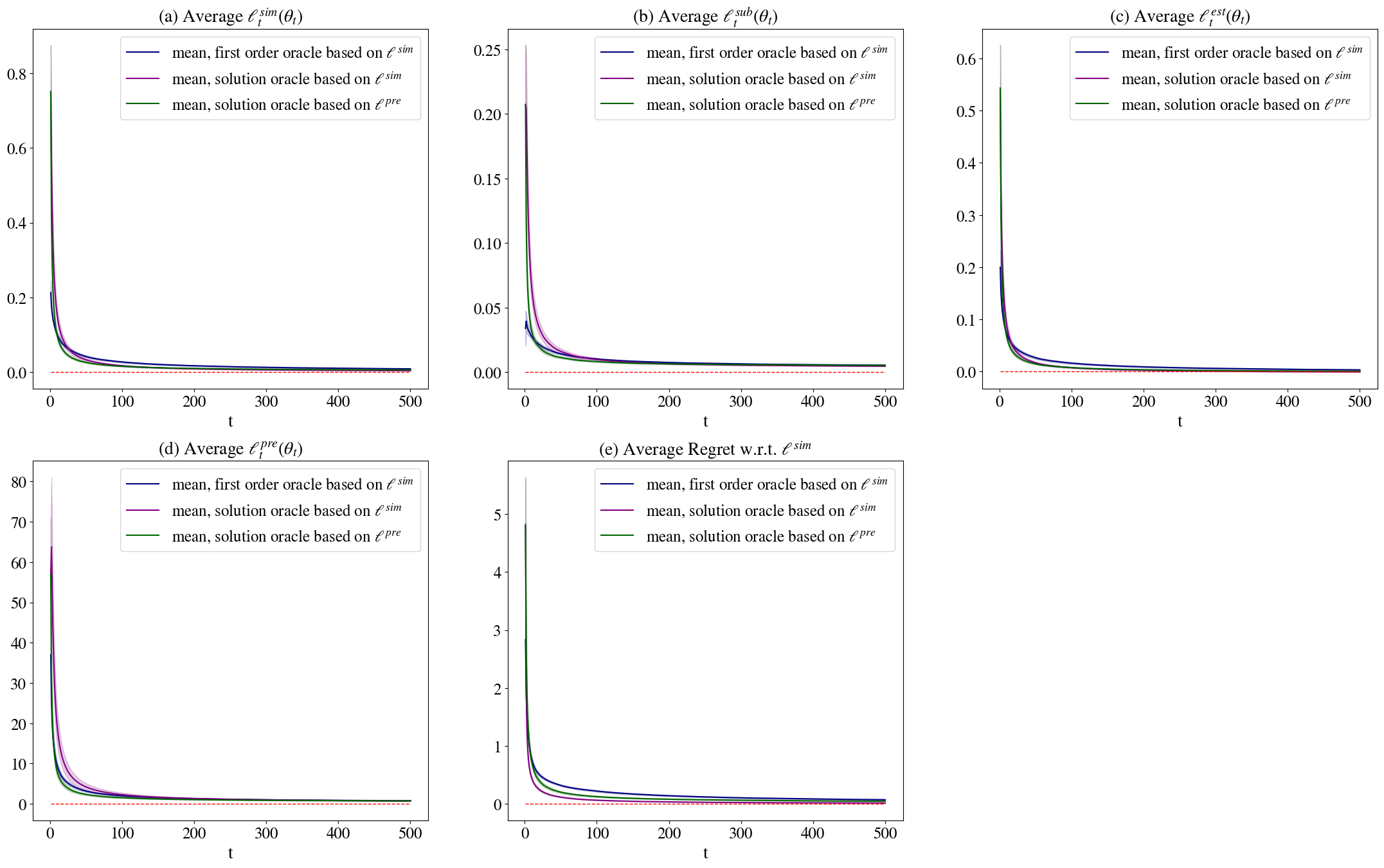}
  \caption{Means of average losses with respect to different loss functions and means of average regret with respect to $\ell^{sim}$.}
  \label{fig:no-knap-50-50}
\end{subfigure}
\begin{subfigure}{0.9\textwidth}
\centering
  \includegraphics[width=0.9\linewidth]{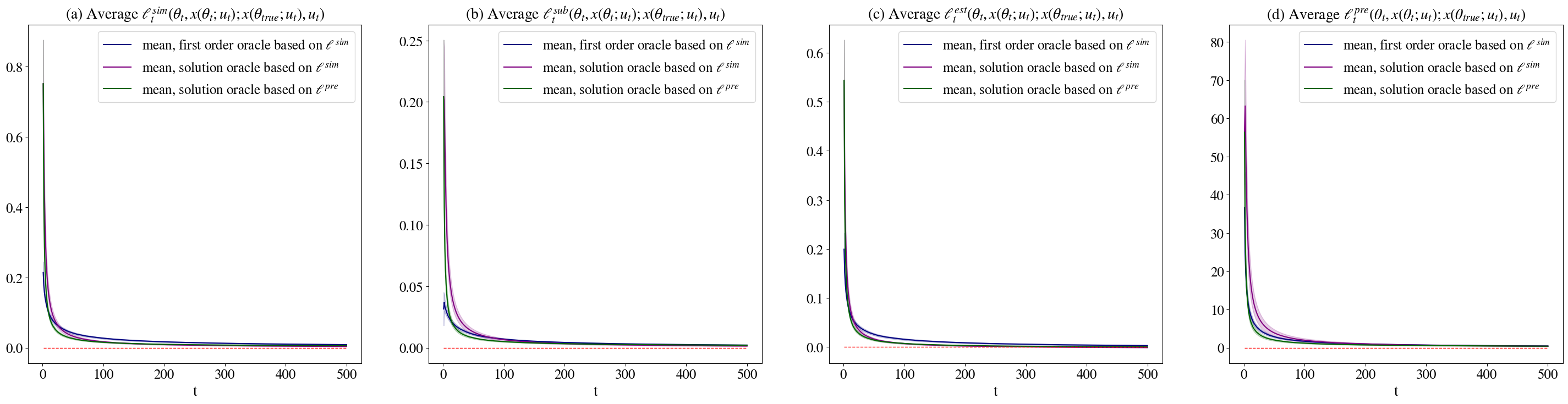}
  \caption{Means of average losses with respect to different loss functions measured at ${x}(\theta_{true};{u}_t)$.}
  \label{fig:no-knap-50-50-per}
\end{subfigure}
\caption{Learning a quadratic utility function under small noises: means of selected performance measures over $T=500$ iterations for continuous knapsack instances; the shaded region is 95\% confidence interval for the means.}
\label{fig-noqua-1}
\end{figure}

\begin{figure}
\centering
\begin{subfigure}{\textwidth}
\centering
\includegraphics[width=0.7\linewidth]{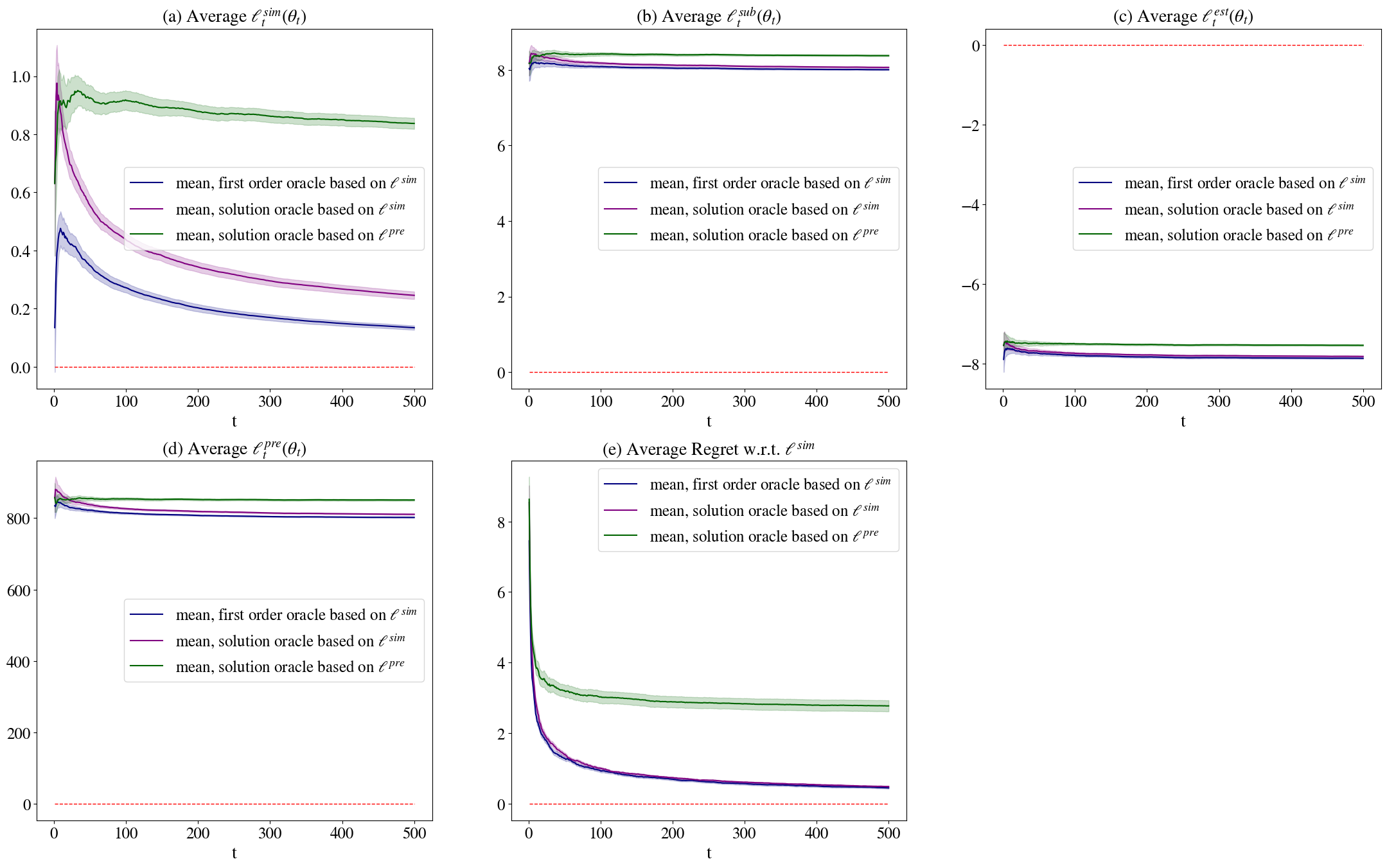}
  \caption{Means of average losses with respect to different loss functions and means of average regret with respect to $\ell^{sim}$.}
  \label{fig:bigno-knap-50-50}
\end{subfigure}
\begin{subfigure}{0.9\textwidth}
\centering
  \includegraphics[width=0.9\linewidth]{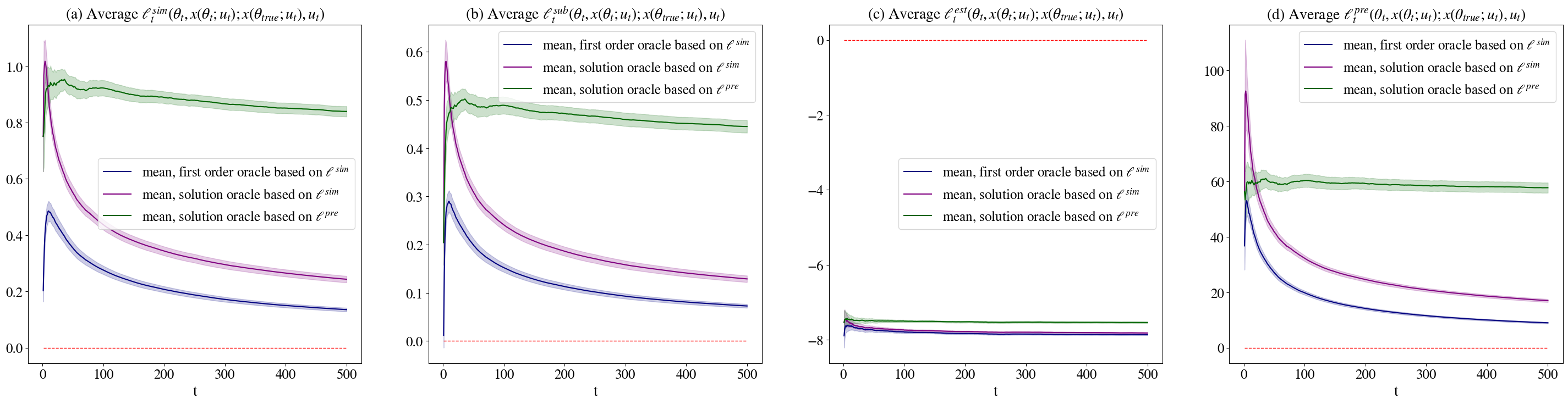}
  \caption{Means of average losses with respect to different loss functions measured at ${x}(\theta_{true};{u}_t)$.}
  \label{fig:bigno-knap-50-50-per}
\end{subfigure}
\caption{Learning a quadratic utility function under large noises: means of selected performance measures over $T=500$ iterations for continuous knapsack instances; the shaded region is 95\% confidence interval for the means.}
\label{fig-noqua-2}
\end{figure}

\subsubsection{Learning a CES Utility Function} \label{sec:app-imperces}
We now provide details on imperfect information experiments in the CES setup, i.e., when the agent's problem has the form~\eqref{ex:prefces} with the equality constrained knapsack domain, i.e.,  $\mathcal{X}(u_t)=\mathcal{X}^{eck}(u_t)$, under small noises in Figures \ref{fig:ces-loss-no}, \ref{fig:ces-loss-no-pm}, and under large noises in Figures \ref{fig:ces-loss-bigno}, \ref{fig:ces-loss-bigno-pm}. Our findings are not fully in line with our observations from learning a quadratic utility function.
The OL algorithm utilizing the first-order oracle is again robust to the small noises, and generate average losses converging in roughly the same patterns as their counterparts in the perfect information case. The OL algorithm with $\ell^{sim}$ based solution oracle, on the other hand, is much less robust to small noises as shown in Figure \ref{fig:ces-loss-no}. As the noises get larger, not surprisingly, the performance of both algorithms degrade. In addition, we note that regardless of the noise magnitude, the solution oracle based OL algorithm has much wider confidence intervals than the first order algorithm. 

\begin{figure}
\centering
\begin{subfigure}{\textwidth}
\centering
  \includegraphics[width=0.7\linewidth]{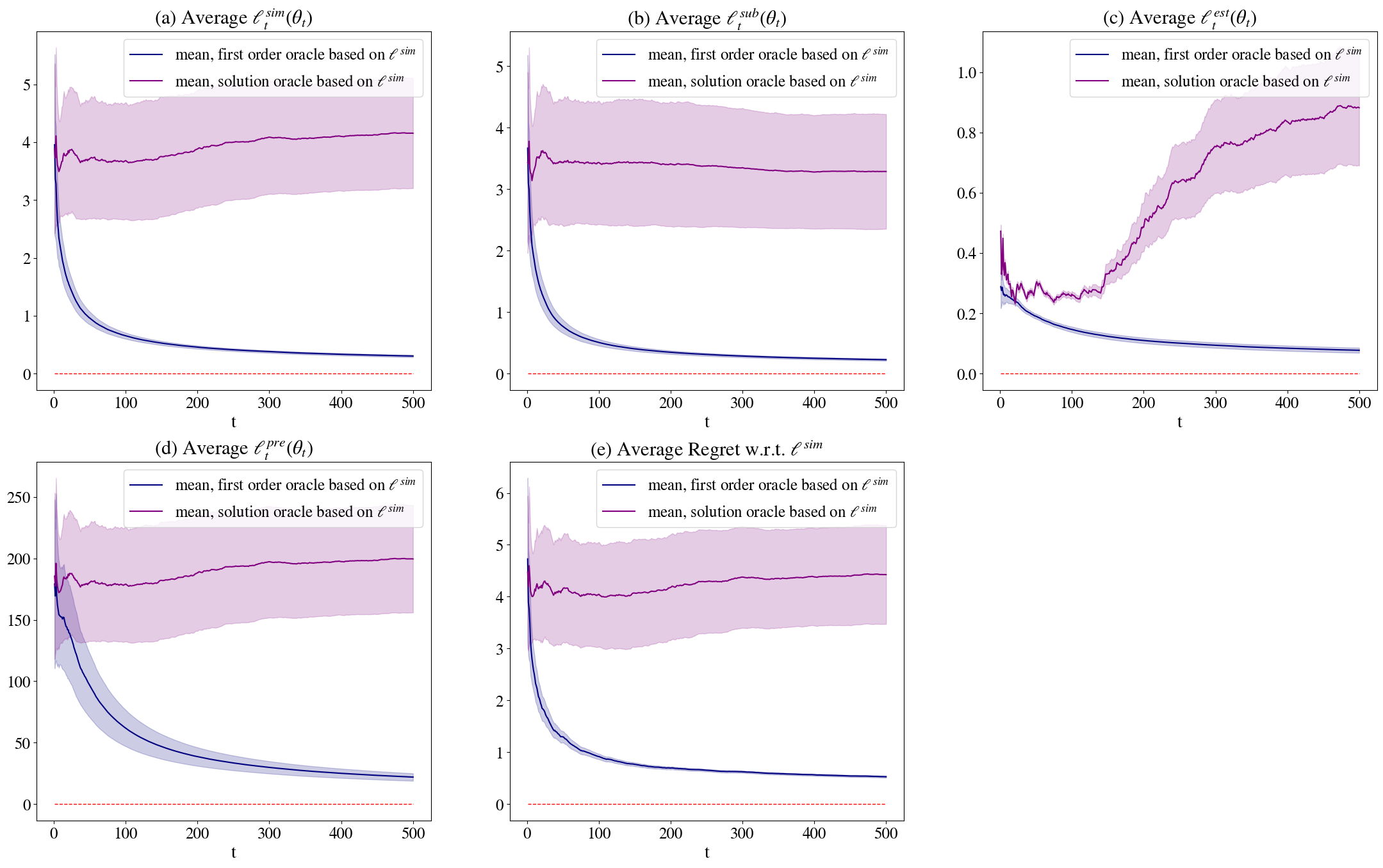}
  \caption{Means of average losses with respect to different loss functions and means of average regret with respect to $\ell^{sim}$.}
  \label{fig:ces-loss-no}
\end{subfigure}

\begin{subfigure}{0.9\textwidth}
\centering
  \includegraphics[width=0.9\linewidth]{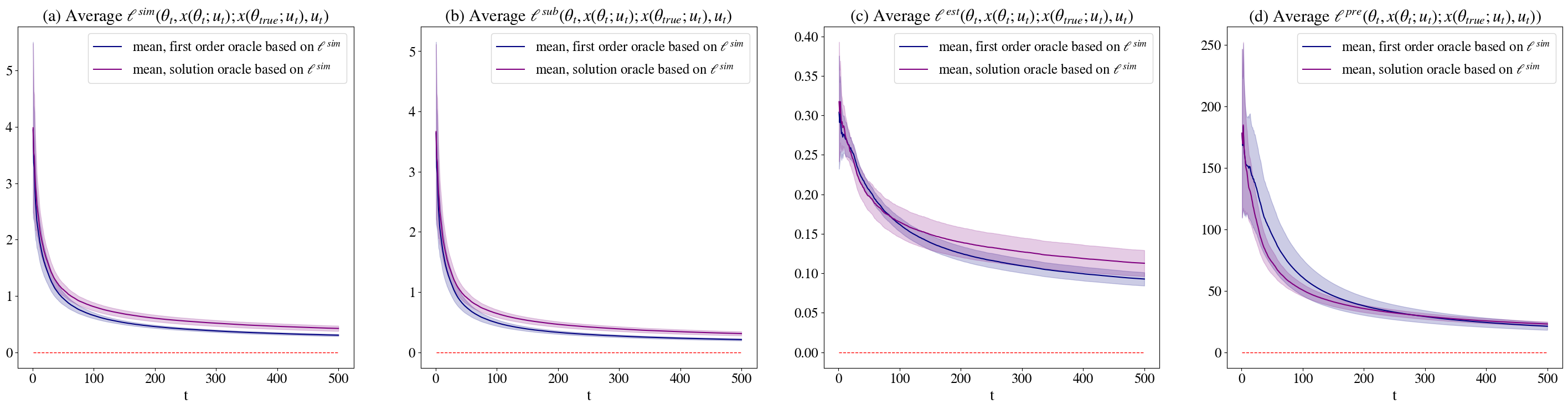}
  \caption{Means of average losses with respect to different loss functions measured at ${x}(\theta_{true};{u}_t)$.}
  \label{fig:ces-loss-no-pm}
\end{subfigure}
\caption{Learning a CES utility function under small noises: means of selected performance measures over $T=500$ iterations for equality constrained knapsack instances with $n=50$; the shaded region is 95\% confidence interval for the means.}
\label{fig:no-ces-1}
\end{figure}

\begin{figure}
\centering
\begin{subfigure}{\textwidth}
\centering
  \includegraphics[width=0.7\linewidth]{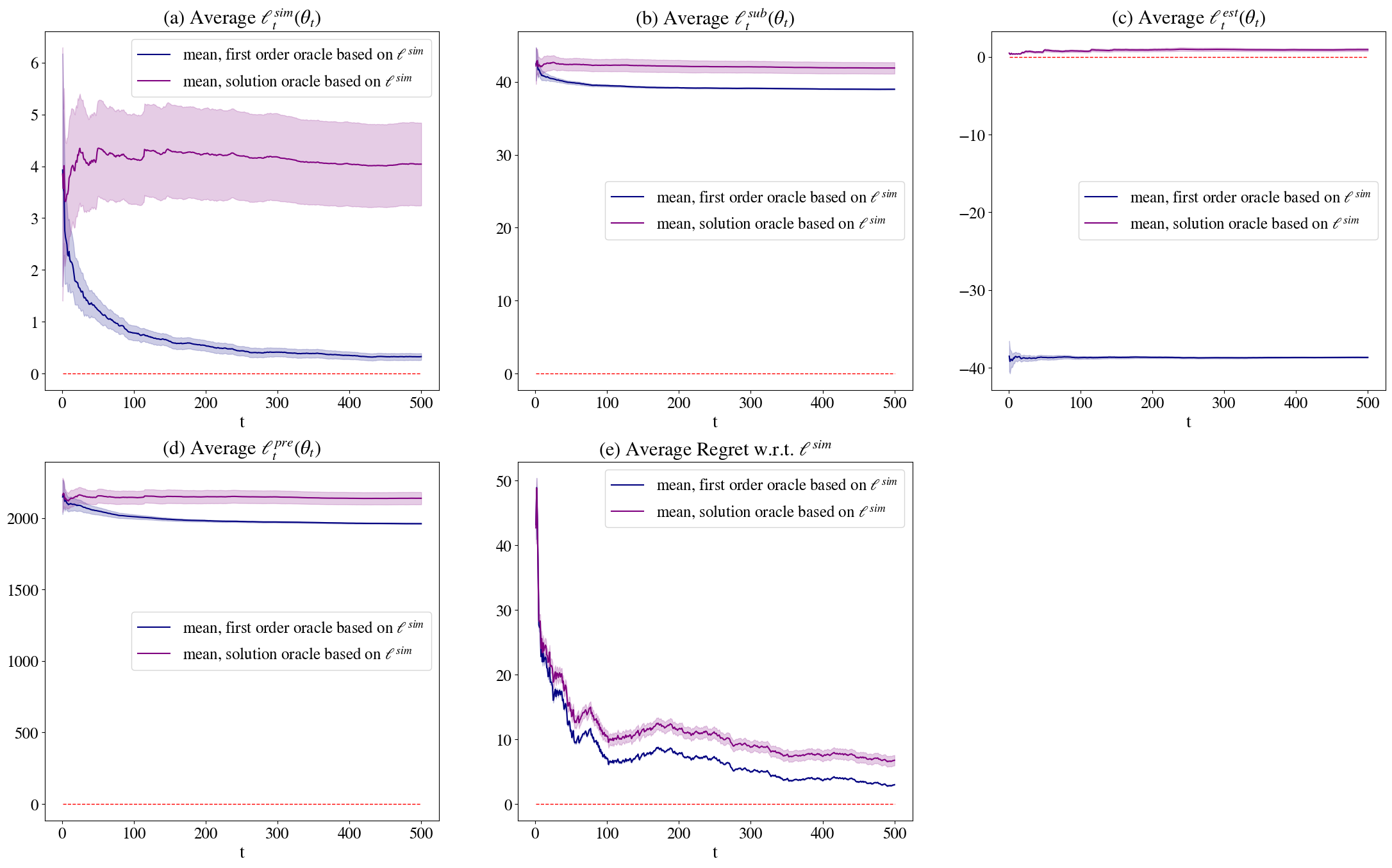}
  \caption{Means of average losses with respect to different loss functions and means of average regret with respect to $\ell^{sim}$.}
  \label{fig:ces-loss-bigno}
\end{subfigure}

\begin{subfigure}{0.9\textwidth}
\centering
  \includegraphics[width=0.9\linewidth]{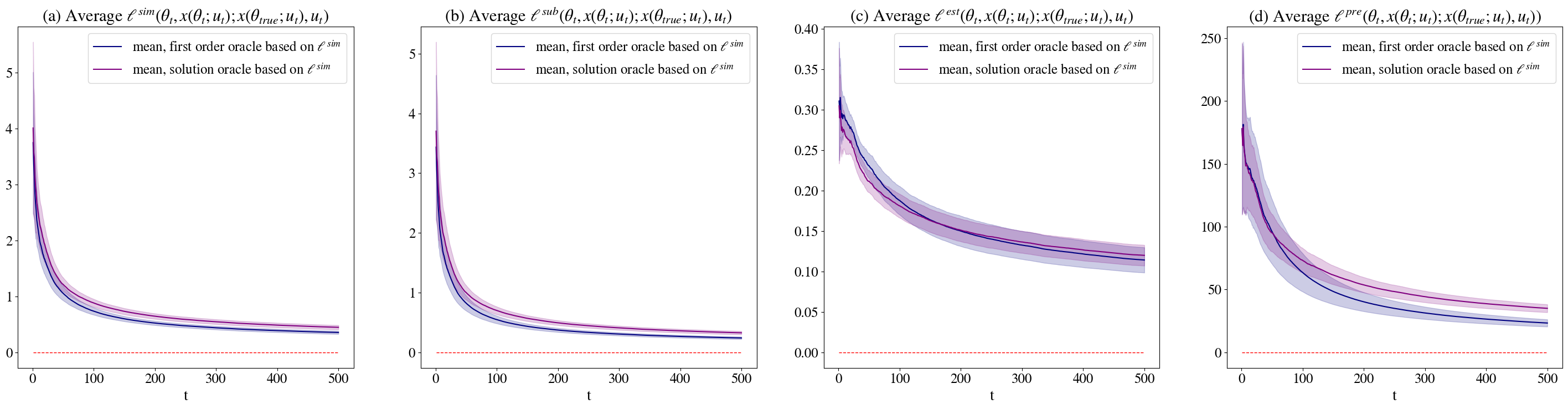}
  \caption{Means of average losses with respect to different loss functions measured at ${x}(\theta_{true};{u}_t)$.}
  \label{fig:ces-loss-bigno-pm}
\end{subfigure}
\caption{Learning a CES utility function under large noises: means of selected performance measures over $T=500$ iterations for equality constrained knapsack instances with $n=50$; the shaded region is 95\% confidence interval for the means.}
\label{fig:no-ces-2}
\end{figure}


\newpage
\bibliographystyle{ACM-Reference-Format}
\bibliography{references}

\newpage
\appendix
\section{CES and Cobb-Douglas Utility Functions} \label{app:utilfunc}
We illustrate how a CES function and a Cobb-Douglas function can be transformed to satisfy our Assumption \ref{assp:fstr}. This then confirms that our online inverse optimization framework is able to handle both types of utility functions.

\subsection{CES Function}
For ${x} \in \R^n_+$, the function $U({x}) \coloneqq (\sum_{i=1}^n a_i x_i^{\rho})^{1/\rho}$ for some $-\infty < \rho < 0$ or $0 < \rho < \infty$ and ${a} \in \R_+^n$ such that $\sum_{i=1}^n a_i = 1$ is referred to as a CES function. An economic interpretation of CES functions is provided in \cite{balcan2014learning}: ${x}$ represents an outcome where the agent receives amount $x_i$ of good $i$, and the utility $U({x})$ captures both substituteness and complementarity of the $n$ goods. Here, for consistency of notation, we replace ${a}$ with $\theta_{true}$. 

$U({x})$ is a concave function of ${x}\in \mathcal{X}\subseteq\R^n_+$ whenever $\rho\in(-\infty,0)\cup(0,1]$, and the agent's forward problem maximizes $U({x})$:
\[
\max_{{x}} \left\{ \left( \sum_{i=1}^n (\theta_{true})_i x_i^{\rho} \right)^{1/\rho}:~ g({x}; {u}) \leq 0,~ {x} \in \mathcal{X} \right\}.
\]

Equivalently, the agent's optimal solution ${x}(\theta_{true}; {u})$ can be obtained with the following systems.
\[
{x}(\theta_{true}; {u}) = \left\{ 
\begin{array}{ll}
\argmin_{{x}} \left\{  \sum_{i=1}^n (\theta_{true})_i x_i^{\rho}:~ g({x}; {u}) \leq 0,~ {x} \in \mathcal{X} \right\},
& -\infty < \rho < 0, \\
\argmin_{{x}} \left\{  -\sum_{i=1}^n (\theta_{true})_i x_i^{\rho}:~ g({x}; {u}) \leq 0,~ {x} \in \mathcal{X} \right\},
& 0 < \rho \leq 1.
\end{array}
\right.
\]

$U({x})$ is a convex function of ${x}\in \mathcal{X}\subseteq\R^n_+$ whenever $\rho\in(1,\infty)$, and the agent's forward problem minimizes $U({x})$:
\begin{align*}
\min_{{x}} & \left\{ \left( \sum_{i=1}^n (\theta_{true})_i x_i^{\rho} \right)^{1/\rho}:~ g({x}; {u}) \leq 0,~ {x} \in \mathcal{X} \right\}, \\
\text{and thus }~ 
{x}(\theta_{true}; {u}) &= %
\begin{array}{ll}
\argmin_{{x}} \left\{  \sum_{i=1}^n (\theta_{true})_i x_i^{\rho}:~ g({x}; {u}) \leq 0,~ {x} \in \mathcal{X} \right\},
& 1 < \rho < \infty.
\end{array} 
\end{align*}
Note that these alternative representations of agent's objective function satisfy Assumption \ref{assp:fstr}.

\subsection{Cobb-Douglas Function}
For ${x} \in \R^n_+$, the function $U({x}) = \Pi_{i = 1}^n x_i^{a_i}$, where $a_i > 0$ and $\sum_{i=1}^n a_i \leq 1$, is referred to as a Cobb-Douglas function; see \cite{roth2016watch}. This utility function can be derived from the CES function by taking $\rho \rightarrow 0$.
We replace ${a}$ with $\theta_{true}$ for consistency, then an agent with the given Cobb-Douglas utility function chooses her/his optimal action ${x}(\theta_{true};{u})$ as:
\[
{x}(\theta_{true};{u}) = \argmax_{{x}} \left\{ \sum_{i=1}^n (\theta_{true})_i \log x_i:~  g({x}; {u}) \leq 0,~ {x} \in \mathcal{X} \right\}.
\]
We obtain this reformulation by taking logarithm of the product form objective. We immediately observe that Assumption \ref{assp:fstr} holds for this transformed representation.

\section{Information Obscuring Agent Objective Example} \label{app:cstruc}
We give a simple 1-dimensional example of an agent's utility function that obscures information due to its particular choice of $c(x)$. Suppose $\theta_{true} = 0$ and $\Theta = [-3,3]$, an agent's forward problem is $\min_x \{x + \theta c(x): x \in [-1,1]\}$ and let $\mathcal{X}(\theta)$ denote the set of optimal solutions to the agent's problem. With a given $\theta$, the predicted agent action is denoted by $x(\theta) \in \mathcal{X}(\theta)$. 

We define 
\[c(x) \coloneqq \begin{cases} 0 & \text{ if } x \neq 0 \\ -1 & \text{ if } x = 0 \end{cases}.\] 
This particular function obscures information on $x$. The agent's objective function simplifies to $x$ when $x \neq 0$ and $x - 2\theta$ when $x = 0$. Since $\theta_{true} = 0$, it is clear that $x(\theta_{true}) = -1$ is the minimizer. For $\theta \neq 0$, the agent's problem is given by $\min\left\{\min_x\{x:~ x\in[-1,1], x\neq 0\},~ 0+\theta*(-1) \right\} =\min \{ -1, -\theta\}$. Then, we deduce that the agent's optimal solution will satisfy the following:
\begin{itemize}
\item When $\theta<0$, agent's optimal solution is $x(\theta) = -1$;
\item When $\theta>0$, $x(\theta) = -1$ if $\theta < 1 $, and $x(\theta) = 0$ if $\theta \geq 1$ (note that in the case of alternative optima, we assume that the solver will break ties by selecting the solution with smaller norm);
\item When $\theta=0$, $x(\theta) = -1$.
\end{itemize}
To summarize, in the given domain $\Theta$, when $\theta \in [-3,1)$, $x(\theta) = -1$ and $c(x(\theta)) = 0$; when $\theta \in [1,3]$, $x(\theta) = 0$ and $c(x(\theta)) = -1$.

We next show that implicit OL based on $\ell^{pre}$ with a solution oracle may lead to an unbounded regret. Suppose we choose $\eta_t = \frac{1}{t}$ for all $t$. Then, at time step $t$, based on the implicit OL based on $\ell^{pre}$, we update $\theta_{t+1}$ by solving the following optimization problem: in this example, $\ell_t^{pre}(\theta) = \norm{x(\theta_{true}) - x(\theta)}^2=(-1-x(\theta))^2$, hence
\[\theta_{t+1} = \argmin_{\theta \in [-3,3]} \frac{1}{2}(\theta-\theta_t)^2 + \frac{1}{t} (-1-x(\theta))^2 .
\]

If we initialize $\theta_1 = 3$, then the above update will generate $\theta_2 = \argmin_{\theta \in [-3,3]} \frac{1}{2}(\theta-3)^2 + \frac{1}{t}(-1 - x(\theta))^2$. To decide the optimal solution, we need to compare three scenarios: when $\theta = \theta_1$, the objective value is $0 + \frac{1}{t}(-1 - x(\theta_1))^2 = \frac{1}{t}(-1 -0)^2 = \frac{1}{t}$; when $\theta < 1$, the objective value is $\frac{1}{2}(\theta - \theta_1)^2 + \frac{1}{t}(-1 - x(\theta))^2  = \frac{1}{2}(\theta - 3)^2 + 0 \geq \frac{1}{2}(1 - 3)^2 = 2 > \frac{1}{t}$ (where we used $x(\theta)=-1$ for $\theta\leq 1$); when $1 \leq \theta < \theta_1$, the objective value is $\frac{1}{2}(\theta - \theta_1)^2 + \frac{1}{t}(-1 - x(\theta))^2  = \frac{1}{2}(\theta - 3)^2 + \frac{1}{t}(-1-0)^2 > \frac{1}{t}$.
Therefore, we have $\theta_2 = \theta_1$, and by the same derivation, later iterations will always stay at the same estimate $\theta_t = \theta_1$. This means the implicit OL algorithm will generate $\theta_t = 3$ for all $t$, and each iteration the learner incurs prediction loss as $\ell_t^{pre}(\theta_t) = \norm{x(\theta_{true})-x(\theta_t)}^2 = \norm{x(0)-x(3)}^2 = 1$. Therefore, the associated regret with respect to $\ell^{pre}$ is unbounded as $T \rightarrow \infty$: 
\[
R_T(\{\ell^{pre}_t \}_{t \in [T]}, \{\theta_t\}_{t \in [T]}) = \sum_{t \in [T]} \ell_t^{pre}(\theta_t) -\sum_{t \in [T]} \ell_t^{pre}(\theta_{true}) = T.
\]

We note that this example does not invalidate the regret convergence in Theorem \ref{thm:IOLRegretBdExplicit}. With the contrived definition of $c(x)$, the loss function $\ell^{pre}$ does not satisfy the Lipschitz continuity assumption needed for regret convergence guarantees given in Theorem \ref{thm:IOLRegretBdExplicit}. To be more specific, $\ell^{pre}_t(\theta) = (x(\theta_{true}) - x(\theta))^2$: consider $\epsilon >0$, $\theta_1 = 1,~\theta_2 = 1 + \epsilon > 1$, we conclude $\ell^{pre}_t(\theta_1) = (-1-(-1))^2 = 0$ and $\ell_t^{pre}(\theta_2) = (-1-0)^2 = 1$. As $\epsilon \rightarrow 0$, there is no finite $G$ as a valid Lipschitz constant for $\ell_t^{pre}$.

We also examine the use of online Mirror Descent (MD) based on $\ell^{sim}$ in the same setup. Let Euclidean distance be the distance generating function in Bregman distance, then online MD simplies to projected gradient descent. We again choose $\eta_t = \frac{1}{t}$ for all $t$, then at time step $t$ we update $\theta_{t+1}$ via 
\[
\theta_{t+1} = \text{proj}_{[-3,3]} \left[\theta_t - \frac{1}{t} \big(c(x(\theta_{true})) - c(x(\theta_t))\big)\right].
\]
Suppose we initialize $\theta_1 = 3$, then $\theta_2 = \text{proj}_{[-3,3]} [\theta_1 - \frac{1}{t}(0 - (-1))] = \theta_1 - \frac{1}{t} = 2$. Following similar derivations, we will update $\theta_t$ in the later iterations as:
\begin{gather*}
    \begin{aligned}
    & \theta_3 = \text{proj}_{[-3,3]} [\theta_2 - \frac{1}{t}(0 - (-1))] = 2 - \frac{1}{2} = \frac{3}{2} \\
    & \theta_4 = \text{proj}_{[-3,3]} [\theta_3 - \frac{1}{t}(0 - (-1))] = \frac{3}{2} - \frac{1}{3} = \frac{7}{6} \\
    & \theta_5 = \text{proj}_{[-3,3]} [\theta_4 - \frac{1}{t}(0 - (-1))] = \frac{7}{6} - \frac{1}{4} = \frac{11}{12} ~ \text{(Note: $c(x(\theta_5)) = 0$)}\\
    & \theta_6 = \text{proj}_{[-3,3]} [\theta_5 - \frac{1}{t}(0 - 0)] = \frac{11}{12}. 
    \end{aligned}
\end{gather*}
It is clear that all later iterations will stay at the same estimate $\theta_t = \frac{11}{12}$ for $t \geq 5$. By definition, $\ell^{sim}_t(\theta_t) = (\theta_t - \theta_{true})\left[c(x(\theta_{true})) - c(x(\theta_t))\right] = \theta_t (0 - (-1)) = \theta_t$ for $t = 1, \ldots, 4$, and $\ell^{sim}_t(\theta_t) = \theta_t (0 - 0) = 0$ for $t \geq 5$. Therefore, the regret with respect to $\ell^{sim}$ becomes 
\[R_T(\{\ell^{sim}_t \}_{t \in [T]}, \{\theta_t\}_{t \in [T]}) = 3 + 2 + 3/2 + 7/6.\]
As $T$ increases, then the average regret converges to $0$. Contrary to the $\ell^{pre}$ based implicit OL algorithm, online MD based on $\ell^{sim}$ leads to converging regret with respect to the loss function of choice. In addition, we point out that this does not violate the bounding relationship between regrets based on $\ell^{sim}$ and $\ell^{pre}$ in Corollary \ref{cor:strconv}, because the agent's objective function in this simple example is not strongly convex. 

\section{Convexity Status of $\ell^{pre}$ and $\ell^{est}$} \label{app: convexity}
In this appendix, we examine the convexity status $\ell^{pre}$ and $ \ell^{est}$ under our Assumption~\ref{assp:fstr}. Recall that we already establish in  Lemma~\ref{lem:l-sub-characterization} that under Assumption~\ref{assp:fstr} $\ell^{sub}(\theta)$ is convex in $\theta$. 
On the other hand, $\ell^{pre}$ and $\ell^{est}$ are not guaranteed to be convex in $\theta$ even under Assumption~\ref{assp:fstr} and even when agent's problem is a one dimensional optimization problem.
\begin{example}
Suppose $\Theta = [-1,1]$ and $\theta_{true} = 1$, and an agent's forward problem is $\min_x \{ \theta x: x \in \mathcal{X} \}$. We consider a convex domain $\mathcal{X} = [-1, 1]$. Let $x(\theta)\coloneqq \argmin_x \{ \theta x: x \in \mathcal{X} \}$, i.e., the set of optimizers of the agent's problem for given $\theta$. Then, we easily deduce that the agent's optimal action(s) at a given $\theta$ are: $x(\theta) = -1$ if $\theta > 0$, $x(\theta) \in \mathcal{X}$ if $\theta=0$, and $x(\theta) = 1$ if $\theta < 0$. Specifically, this implies $x(\theta_{true}) = -1$.  

Consider $\theta_1 = 1$, $\theta_2 = -1$ and $\lambda = \frac{1}{4}$, then $\tilde{\theta} = \lambda\theta_1 + (1-\lambda)\theta_2 =  -\frac{1}{2}$. By the format of $x(\theta)$ in this problem and the definition of $\ell^{est}$, we observe that 
\begin{align*}
    &\ell^{est} (\theta_1) = \theta_{true} (x(1) - x(\theta_{true})) = 0,\\
    &\ell^{est}(\theta_2) = \theta_{true} (x(-1) - x(\theta_{true})) = 2,\\
    &\ell^{est}(\tilde{\theta}) = \theta_{true} (x(-1/2) - x(\theta_{true})) = 2. 
\end{align*}
Therefore, we deduce $\ell^{est}(\tilde{\theta}) > \lambda \ell^{est} (\theta_1) + (1-\lambda) \ell^{est} (\theta_2)$ that shows that $\ell^{est}$ is not a convex function of $\theta$. 
Similarly, in the case of $\ell^{pre}$, we arrive at
\begin{align*}
    & \ell^{pre} (\theta_1) = (x(1) - x(\theta_{true}))^2 = 0, \\
    &\ell^{pre}(\theta_2) = (x(-1) - x(\theta_{true}))^2 = 4, \\
    &\ell^{pre}(\tilde{\theta}) = (x(-1/2) - x(\theta_{true}))^2 = 4.
\end{align*}
Similarly, we arrive at $\ell^{pre}(\tilde{\theta}) > \lambda \ell^{pre} (\theta_1) + (1-\lambda) \ell^{pre} (\theta_2)$ and hence conclude $\ell^{pre}$ is not convex.
\end{example}
Note that the nonconvexity of $\ell^{est}$ and $\ell^{pre}$ established in this example remains the same even if we switch to an integer domain of $\mathcal{X} = \{-1, 1\}$.

\section{Proofs} \label{sec:proof}

\subsection{Proof of Lemma~\ref{lem:l-sub-characterization}} \label{app:lsubproof}
By definition of $\ell^{sub}(\theta)$, we have
\begin{align*}
    & \ell^{sub}(\theta, {x}(\theta;{u}_t); {y}_t, {u}_t)
    = f({y}_t; \theta, {u}_t) - f({x}(\theta;{u}_t); \theta, {u}_t) \\
    &= f_1({y}_t;{u}_t) + f_2(\theta;{u}_t) + \langle \theta, c({y}_t) \rangle
    -f_2(\theta;{u}_t) - \min_x\{f_1({x};{u}_t) + \langle \theta, c({x}) \rangle:~ g(x;u_t)\leq 0\}\\
    &= f_1({y}_t;{u}_t) + \langle \theta, c({y}_t) \rangle - \min_x\{f_1({x};{u}_t) + \langle \theta, c({x}) \rangle:~ g(x;u_t)\leq 0\} \\
    &= f_1({y}_t;{u}_t) + \max_x\{\langle \theta, c({y}_t)-c({x}) \rangle-f_1({x};{u}_t):~ g(x;u_t)\leq 0\}.
\end{align*}
Here, the second equation follows from Assumption~\ref{assp:fstr} and the remaining equations are simply cancellation and re-arrangements of the terms. 
Thus, under Assumption~\ref{assp:fstr}, $\ell^{sub}(\theta)$ is a point-wise maximum of affine functions of $\theta$ and hence it is a convex function of $\theta$. \qed

\subsection{Proof of Lemma \ref{lem:l-sim_convex}}
When Assumption \ref{assp:fstr} holds, based on the given form of $f$, $\ell_t^{sim}(\theta)$ simplifies to a function linear in $\theta$, hence is convex with respect to $\theta$. 
\qed

\subsection{Proof of Proposition \ref{prop:pi}}
We break down the proof of Proposition \ref{prop:pi} into several intermediate results.

We first observe some important properties of $\ell^{sim}_t$ and its connection with $\ell^{sub}_t, \ell^{est}_t$. These properties play a key role in the development of our regret guarantees.
\begin{lemma} \label{lem:simloss}
Suppose Assumption \ref{assp:fstr} holds. Then, for $t \in [T]$ and any signal ${u}_t$, we have
\begin{enumerate}
    \item[(a)] $\ell^{sim}_t(\theta_{true}) = 0$;
    \item[(b)] $\ell^{sim}_t(\theta) + \la \theta - \theta_{true},  c({x}(\theta_t;{u}_t)) -  c({x}(\theta;{u}_t)) \ra = \ell^{sub}_t(\theta) + \ell^{est}_t(\theta) $ for all $\theta$;
    \item[(c)] $\ell^{sim}_t(\theta_t) = \ell^{sub}_t(\theta_t) + \ell^{est}_t(\theta_t)$ for all $t$. 
\end{enumerate}
\end{lemma}
\begin{proof}
Part $(a)$ is evident from the definition of $\ell^{sim}$. Part $(b)$ follows from evaluating these loss functions at a given $\theta$ under Assumption \ref{assp:fstr}:
\begin{gather*}
    \begin{aligned}
\ell_t^{sub}(\theta) + \ell_t^{est}(\theta) & = ( f_1(y_t) + f_2(\theta) + \langle \theta, c(y_t) \rangle - f_1(x(\theta;u_t)) - f_2(\theta) - \langle \theta, c(x(\theta;u_t)) \rangle) \\
& +  (f_1(x(\theta;u_t)) + f_2(\theta_{true}) + \langle \theta_{true}, c(x(\theta;u_t)) \rangle -  f_1(y_t) - f_2(\theta_{true}) - \langle \theta_{true}, c(y_t) \rangle ) \\
& = \langle \theta, c(y_t) - c(x(\theta;u_t))\rangle + \langle \theta_{true}, c(x(\theta;u_t))  -c(y_t)\rangle  \\
& = \langle \theta - \theta_{true}, c(y_t) - c(x(\theta;u_t)) \rangle \\
& = \ell_t^{sim}(\theta) + \langle \theta - \theta_{true},  c(x(\theta_t; u_t)) - c(x(\theta;u_t)) \rangle
    \end{aligned}
\end{gather*}
Finally, part (c) follows from the fact that when we replace $\theta$ with $\theta_t$ in $(b)$, the term representing the difference between $\ell^{sim}_t(\theta)$ and $\ell^{sub}_t(\theta) + \ell^{est}_t(\theta)$ is equal to $0$. 
\end{proof}


We next state a simple observation on the properties of the loss functions.
\begin{observation}\label{obs:loss_nonnegativity}
For every $t$, we have 
\begin{enumerate}[(a)]
    \item $\ell_t^{pre}(\theta)$, is a nonnegative function of $\theta$,
    \item $\ell_t^{sub}(\theta)$ and $\ell_t^{est}(\theta)$ are nonnegative functions of $\theta$ whenever there is no noise, i.e., $y_t=x(\theta_{true};u_t)$.
\end{enumerate}
\end{observation}
\begin{proof}
In part (a), the non-negativity of $\ell_t^{pre}$ is obvious from its squared-norm definition. In part (b), $\ell^{sub}_t(\theta)$ is nonnegative because the objective function value of \eqref{sys:forward-theta} at a feasible solution $y_t = x(\theta_{true}; u_t)$ is no smaller than the optimal objective value at an optimal solution $x(\theta; u_t)$. By a similar argument, $\ell^{est}_t(\theta)$ is nonnegative because $y_t = x(\theta_{true}; u_t)$ is an optimal solution and $x(\theta;u_t)$ is a feasible solution to the forward problem \eqref{sys:forward}.
\end{proof}

Observation~\ref{obs:loss_nonnegativity} leads to the following result that is instrumental in simplifying the regret terms in the noiseless case.
\begin{lemma} \label{lem:minloss}
Consider the noiseless case, i.e., ${y}_t = {x}(\theta_{true},{u}_t)$ for all $t$. Let $\ell_t$ denote any one of the loss functions $\ell_t^{pre}, \ell_t^{sub}$, or $\ell_t^{est}$. Then, 
\begin{enumerate}[(a)]
    \item $\argmin_{\theta \in \Theta} \sum_{t \in [T]} \ell_t (\theta) = \theta_{true}$,
    and $0 = \min_{\theta \in \Theta} \sum_{t \in [T]} \ell_t (\theta)$,
    \item $R_T(\{\ell_t\}_{t \in [T]}, \{\theta_t\}_{t \in [T]}) = \sum_{t \in [T]} \ell_t(\theta_t)$.
\end{enumerate}
\end{lemma}
\begin{proof}
By Observation~\ref{obs:loss_nonnegativity}, we have $\sum_{t \in [T]} \ell_t (\theta) \geq 0$ for all $\theta$. In addition, from evaluating these loss functions at $\theta = \theta_{true}$ in the noiseless case, we have $\ell_t(\theta_{true})= 0$ for all $t$, therefore $\sum_{t \in [T]} \ell_t (\theta_{true}) = 0$. This then implies $\min_{\theta \in \Theta} \sum_{t \in [T]} \ell_t (\theta) = 0$ and the minimum is achieved at $\theta_{true}$, proving Part (a). Part (b) follows from the definition of the regret and Part (a).
\end{proof}

We are now ready to prove Proposition~\ref{prop:pi}.
\begin{proof}[Proof of Proposition~\ref{prop:pi}]
$~$
\begin{enumerate}
\item[(a)] Let $\ell_t$ represent any of the loss functions $\ell_t^{sub}$, $\ell_t^{est}$, and $\ell_t^{pre}$. In the noiseless case, from Observation~\ref{obs:loss_nonnegativity} we deduce that $\ell_t$ is a nonnegative function of $\theta$. Then, from 
Lemma~\ref{lem:minloss}(b), we conclude that all of the corresponding regret terms, i.e.,  $R_T(\{\ell^{sub}_t\}_{t \in [T]}, \{\theta_t\}_{t \in [T]})$,\\ $R_T(\{\ell^{est}_t\}_{t \in [T]}, \{\theta_t\}_{t \in [T]})$ and $R_T(\{\ell^{pre}_t\}_{t \in [T]}, \{\theta_t\}_{t \in [T]})$ are nonnegative. 

\item[(b)]
From the definition of regret, Lemma~\ref{lem:simloss} and Lemma~\ref{lem:minloss}, we have
\begin{align*}
& R_T(\{\ell^{sub}_t\}_{t \in [T]}, \{\theta_t\}_{t \in [T]}) + R_T(\{\ell^{est}_t\}_{t \in [T]}, \{\theta_t\}_{t \in [T]}) \\
& = \left(\sum_{t \in [T]} \ell_t^{sub}(\theta_t) -0\right)+ \left(\sum_{t \in [T]} \ell_t^{est}(\theta_t) - 0\right) = \sum_{t \in [T]} \ell_t^{sim} (\theta_t)
\end{align*}
where the first equation follows from Lemma~\ref{lem:minloss}(a), and the second equation follows from Lemma~\ref{lem:simloss}(c). 

\item[(c)] By definition of regret term $R_T(\{\ell^{sim}_t\}_{t \in [T]}, \{\theta_t\}_{t \in [T]})$, we have
\begin{gather*}
\begin{aligned}
R_T(\{\ell^{sim}_t\}_{t \in [T]}, \{\theta_t\}_{t \in [T]}) &= \sum_{t \in [T]} \ell_t^{sim}(\theta_t) - \min_{\theta \in \Theta} \sum_{t \in [T]} \ell_t^{sim}(\theta) \\
& \geq \sum_{t \in [T]} \ell_t^{sim}(\theta_t) - \sum_{t \in [T]} \ell_t^{sim}(\theta_{true})\\
& = \sum_{t \in [T]}  \ell_t^{sim}(\theta_t) \\
& =  R_T(\{\ell^{sub}_t\}_{t \in [T]}, \{\theta_t\}_{t \in [T]}) + R_T(\{\ell^{est}_t\}_{t \in [T]}, \{\theta_t\}_{t \in [T]}),
\end{aligned}
\end{gather*}
where the inequality follows from $\sum_{t \in [T]} \ell_t^{sim}(\theta_{true}) \geq \min_{\theta \in \Theta}\sum_{t \in [T]} \ell_t^{sim}(\theta)$, the second equation follows from Lemma~\ref{lem:simloss}(a), and the last equation follows from Part (b). 
\qedhere
\end{enumerate}
\end{proof}

\subsection{Proof of Corollary \ref{cor:strconv}}
It was shown in \cite[Proposition 2.5]{Kuhn2018} that when $f$ is strongly convex in ${x}$ with parameter $\gamma$,  we have $\ell^{sub}_t(\theta) \geq \frac{\gamma}{2}\ell^{pre}_t(\theta)$ for all $t$ and for all $\theta \in \Theta$. Then, using Lemma \ref{lem:minloss}, we deduce $R_T(\{\ell^{sub}_t\}_{t \in [T]},\{\theta_t\}_{t \in [T]}) \geq \frac{\gamma}{2} R_T(\{\ell^{pre}_t\}_{t \in [T]},\{\theta_t\}_{t \in [T]})$.
\qed

\section{Regret Performance under Imperfect Information} \label{app:imperinfo}


In this appendix, we look into the regret performances from OL algorithms under imperfect information and the corresponding implications for learning performance under \cref{assp:fstr}. Since the presence of imperfect information does not affect the convexity of $\ell^{sim}$ and $\ell^{sub}$, OCO algorithms based on either loss function can still attain the respective regret performance guarantees. Specifically, $\ell^{sim}$-based OCO algorithms provide regret bounds with respect to $\ell^{sim}$; $\ell^{sub}$-based OCO algorithms provide regret bounds with respect to $\ell^{sub}$, which further bounds the regret with respect to $\ell^{pre}$ whenever the agent has a strongly convex objective function. The implications of regret bounds for learning performance, however, will be different due to the imperfect information. As discussed in Remark 3, under perfect information, sublinear regret bounds with respect to $\ell^{sim}$ implies subliear regret bounds with respect to $\ell^{sub}$, which then implies that the learned estimates $\{\theta_t\}$ have vanishing errors in terms of the agent's true objective values. Moreover, when a sublinear regret bound with respect to $\ell^{pre}$ is available, this implies the error in the prediction $x(\theta_t; u_t)$ of the agent's true actions vanishes as well. We next study these three types of regrets under imperfect information. Note that we skip the $\ell^{est}$-based regret implications under imperfect information here because we do not have any OL algorithm that bounds only the regret with respect to $\ell^{est}$ on its own. 

\subsection{Regret performance with respect to $\ell^{sub}$ and implications}
For regret with respect to $\ell^{sub}$, we have
\begin{align*}
    R_T^{sub}(\{\theta_t\}_{t \in [T]}) & = \sum_{t=1}^T [f(y_t; \theta_t, u_t) - f(x(\theta_t; u_t); \theta_t, u_t)] - \min_{\theta \in \Theta}\sum_{t=1}^T[f(y_t; \theta, u_t) - f(x(\theta; u_t); \theta, u_t)] .
\end{align*}
Under perfect information, the minimization term in the  $R_T^{sub}(\{\theta_t\}_{t \in [T]})$ reduces to $0$, so we deduce that a vanishing average regret implies a vanishing error in the predicted objective values:
\[
\frac{1}{T} R_T^{sub}(\{\theta_t\}_{t \in [T]}) \rightarrow 0 ~\Rightarrow~ \frac{1}{T} \sum_{t=1}^T f(x(\theta_t; u_t); \theta_t, u_t) \rightarrow \frac{1}{T}\sum_{t=1}^T f(x(\theta_{true}; u_t); \theta_t, u_t) .
\]

Under imperfect information, we cannot eliminate the minimization term in the  $R_T^{sub}(\{\theta_t\}_{t \in [T]})$ definition in the same way. Instead, we can only derive lower and upper bounds on the total errors in the predicted objective values under an additional assumption about the imperfect information. We suppose $f(y_t; \theta_t, u_t) - f(x(\theta_{true}; u_t); \theta_t, u_t) \in [-\epsilon_t, \epsilon_t]$ for some non-negative constant $\epsilon_t$. We can consider $\epsilon_t$ as an indicator of the 'degree of noise' in the observation $y_t$.

For ease of notation, let $d_t \coloneqq f(x(\theta_{true}; u_t); \theta_t, u_t) - f(x(\theta_t; u_t); \theta_t, u_t)$ denote the objective error at step $t$, in other words, $d_t$ is the objective value difference between the optimum decision with respect to the true parameter $\theta_{true}$ and the estimated parameter $\theta_t$ at time step $t$. Before deriving the bounds, we first simplify $R_T^{sub}$ for the specific form of $f$ from \cref{assp:fstr}.

\begin{lemma}
Suppose Assumption \ref{assp:fstr} holds, then we have
\begin{gather} \label{eq:Rsubimper}
    \begin{aligned}
    \sum_{t=1}^T d_t = & R_T^{sub} (\{ \theta_t \}_{t \in [T]}) + \min_{\theta \in \Theta}\sum_{t=1}^T[\langle \theta - \theta_t, c(y_t) \rangle - f_1(x(\theta; u_t)) - \langle \theta, c(x(\theta; u_t)) \rangle] \\
    & + \sum_{t=1}^T [ f_1(x(\theta_{true}; u_t)) + \langle \theta_t, c(x(\theta_{true}; u_t)) \rangle ] .
    \end{aligned}
\end{gather}
\end{lemma}
\begin{proof} 
By the regret definition,
\begin{align*} 
    &R_T^{sub}(\{\theta_t\}_{t \in [T]})\\ 
    & =\sum_{t=1}^T [f_1(y_t) + \langle \theta_t, c(y_t) \rangle - f_1(x(\theta_t; u_t)) - \langle \theta_t, c(x(\theta_t; u_t)) \rangle ] \\
    & ~~ - \min_{\theta \in \Theta}\sum_{t=1}^T[f_1(y_t) + \langle \theta, c(y_t) \rangle - f_1(x(\theta; u_t)) - \langle \theta, c(x(\theta; u_t)) \rangle] \\ 
    & = -\sum_{t=1}^T [ f_1(x(\theta_t; u_t)) + \langle \theta_t, c(x(\theta_t; u_t)) \rangle ] - \min_{\theta \in \Theta}\sum_{t=1}^T[\langle \theta - \theta_t, c(y_t) \rangle - f_1(x(\theta; u_t)) - \langle \theta, c(x(\theta; u_t)) \rangle] \\
     & = \sum_{t=1}^T [ f_1(x(\theta_{true}; u_t)) + \langle \theta_t, c(x(\theta_{true}; u_t)) \rangle ] -\sum_{t=1}^T [ f_1(x(\theta_t; u_t)) + \langle \theta_t, c(x(\theta_t; u_t)) \rangle ] \\
     & ~~ -\min_{\theta \in \Theta}\sum_{t=1}^T[\langle \theta - \theta_t, c(y_t) \rangle - f_1(x(\theta; u_t)) - \langle \theta, c(x(\theta; u_t)) \rangle] - \sum_{t=1}^T [ f_1(x(\theta_{true}; u_t)) + \langle \theta_t,c(x(\theta_{true}; u_t))\rangle ] \\
     & = \sum_{t=1}^T d_t -\min_{\theta \in \Theta}\sum_{t=1}^T[\langle \theta - \theta_t, c(y_t) \rangle - f_1(x(\theta; u_t)) - \langle \theta, c(x(\theta; u_t)) \rangle] - \sum_{t=1}^T [ f_1(x(\theta_{true}; u_t)) + \langle \theta_t,c(x(\theta_{true}; u_t))\rangle ]
\end{align*}
We easily obtain \eqref{eq:Rsubimper} by rearranging the terms.
\end{proof}

\begin{lemma}
Suppose Assumption \ref{assp:fstr} holds and $f(y_t; \theta_t, u_t) - f(x(\theta_{true}; u_t); \theta_t, u_t) \in [-\epsilon_t, \epsilon_t]$, then we have the following bounds on $\sum\nolimits_{t=1}^T d_t$:
\[ \sum_{t=1}^T d_t \leq R_T^{sub} (\{ \theta_t \}_{t \in [T]}) + \sum_{t=1}^T\langle \theta_{true} - \theta_t, c(y_t) - c(x(\theta_{true}; u_t))\rangle; \]
\[  \sum_{t=1}^T d_t \geq R_T^{sub} (\{ \theta_t \}_{t \in [T]}) +  \min_{\theta \in \Theta}\sum_{t=1}^T \ell^{sub}_t(\theta)- \sum_{t=1}^T \epsilon_t .\]
\end{lemma}
\begin{proof}
We derive both formulas by bounding the last two terms in the right hand side of \eqref{eq:Rsubimper}. First, we conclude an upper bound by plugging in $\theta_{true}$ in the minimization term
\begin{align*}
 \sum_{t=1}^T d_t \leq & R_T^{sub} (\{ \theta_t \}_{t \in [T]}) + \sum_{t=1}^T[\langle \theta_{true} - \theta_t, c(y_t) \rangle - f_1(x(\theta_{true}; u_t)) - \langle \theta_{true}, c(x(\theta_{true}; u_t)) \rangle] \\
 & + \sum_{t=1}^T [ f_1(x(\theta_{true}; u_t)) + \langle \theta_t, c(x(\theta_{true}; u_t)) \rangle ] \\
 = &  R_T^{sub} (\{ \theta_t \}_{t \in [T]}) + \sum_{t=1}^T\langle \theta_{true} - \theta_t, c(y_t) - c(x(\theta_{true}; u_t))\rangle .
\end{align*}
For the lower bound, we simply substitute the last two terms in the right hand side of \eqref{eq:Rsubimper} with the noise parameter $-\epsilon_t$.
\begin{align*}
 \sum_{t=1}^T d_t = & R_T^{sub} (\{ \theta_t \}_{t \in [T]}) +  \min_{\theta \in \Theta}\sum_{t=1}^T[\langle \theta, c(y_t) \rangle + f_1(y_t) - f_1(x(\theta; u_t)) - \langle \theta, c(x(\theta; u_t)) \rangle] \\
    & + \sum_{t=1}^T [ f_1(x(\theta_{true}; u_t)) + \langle \theta_t, c(x(\theta_{true}; u_t)) \rangle ] - \sum_{t=1}^T [ f_1(y_t) + \langle \theta_t, c(y_t) \rangle ] \\
    \geq & R_T^{sub} (\{ \theta_t \}_{t \in [T]}) +  \min_{\theta \in \Theta}\sum_{t=1}^T \ell^{sub}_t(\theta)- \sum_{t=1}^T \epsilon_t .
\end{align*}
\end{proof}
When the average regret vanishes, namely $\frac{1}{T}R_T^{sub}(\theta_t) \rightarrow 0$, we have the derived bounds on $\frac{1}{T}\sum_{t=1}^T d_t$ converge as below:
\[
\frac{1}{T}\min_{\theta \in \Theta}\sum_{t=1}^T \ell^{sub}_t(\theta)- \frac{1}{T} \sum_{t=1}^T \epsilon_t \leq \frac{1}{T} \sum_{t=1}^T d_t \leq \frac{1}{T} \sum_{t=1}^T\langle \theta_{true} - \theta_t, c(y_t) - c(x(\theta_{true}; u_t))\rangle .
\]
By inspecting the lower and upper bound expressions, we note that this derived range for $\frac{1}{T} \sum_{t=1}^T d_t = \frac{1}{T} \sum_{t=1}^T f(x(\theta_{true}; u_t); \theta_t, u_t) - f(x(\theta_t; u_t); \theta_t, u_t)$ simplifies to $[0,0]$ when $y_t = x(\theta_{true}; u_t)$ and tends to be smaller when $y_t$ is 'closer' to the true action $x(\theta_{true}; u_t)$.

\subsection{Regret performance with respect to $\ell^{pre}$ and implications}
When the agent's objective function is strongly convex, a sublinear bound on $R_T^{sub} (\{ \theta_t \}_{t \in [T]})$ also implies a sublinear regret bound with respect to $\ell^{pre}$. Recall the following interpretation of a vanishing $\frac{1}{T}R_T^{pre} (\{ \theta_t \}_{t \in [T]})$ under perfect information 
\[
\frac{1}{T}R_T^{pre} (\{ \theta_t \}_{t \in [T]}) \rightarrow 0 ~\Rightarrow~ \frac{1}{T} \sum_{t=1}^T \norm{x(\theta_t; u_t) - x(\theta_{true}; u_t)}^2 \rightarrow \frac{1}{T} \min_{\theta \in \Theta} \sum_{t=1}^T \norm{x(\theta; u_t) - x(\theta_{true}; u_t)}^2 = 0.
\]
By letting $r_t\coloneqq \norm{x(\theta_t; u_t) - x(\theta_{true}; u_t)}$, we deduce that under perfect information $\frac{1}{T}R_T^{pre} (\{ \theta_t \}_{t \in [T]}) \rightarrow 0$ implies that $\frac{1}{T}\sum_{t=1}^T r_t^2 \to 0$ as well.
We now examine the regret implications for $\frac{1}{T}\sum_{t=1}^T r_t^2$ under imperfect information. Similar to our analysis for the $\ell^{sub}$-based regret, we proceed by deriving bounds for this term.
\begin{lemma}
Suppose Assumption \eqref{assp:fstr} holds and $r_t = \norm{x(\theta_t; u_t) - x(\theta_{true}; u_t)}$ for all $t$, then we have
\[\sum_{t=1}^T r_t^2 \leq R_T^{pre} (\{ \theta_t \}_{t \in [T]}) + 2\sum_{t=1}^T r_t \norm{x(\theta_{true}; u_t) - y_t}.\]
\end{lemma}
\begin{proof}
We first rewrite the definition of $R_T^{pre} (\{ \theta_t \}_{t \in [T]})$.
\begin{align*}
    R_T^{pre} (\{ \theta_t \}_{t \in [T]}) &=  \sum_{t=1}^T \norm{x(\theta_t; u_t) - y_t}^2  - \min_{\theta \in \Theta} \sum_{t=1}^T \norm{x(\theta; u_t) - y_t}^2 \\
    & = \sum_{t=1}^T \norm{x(\theta_t; u_t) - x(\theta_{true}; u_t)}^2 + \sum_{t=1}^T \norm{x(\theta_{true}; u_t) - y_t}^2 \\
    & \qquad + 2\sum_{t=1}^T \langle x(\theta_t;u_t) - x(\theta_{true}; u_t), x(\theta_{true}; u_t) - y_t \rangle  - \min_{\theta \in \Theta} \sum_{t=1}^T \norm{x(\theta; u_t) - y_t}^2 \\
    &=  \sum_{t=1}^T r_t^2 - \min_{\theta \in \Theta} \sum_{t=1}^T \norm{x(\theta; u_t) - y_t}^2 \\
    & \qquad + \sum_{t=1}^T \langle 2x(\theta_t; u_t) - x(\theta_{true}; u_t) - y_t,~ x(\theta_{true}; u_t) - y_t\rangle ,
\end{align*}
where the last equation follows from rearranging the terms and plugging in the definition of $r_t$. Therefore, we have 
\begin{align*}
\sum_{t=1}^T r_t^2  
&=  R_T^{pre} (\{ \theta_t \}_{t \in [T]}) + \min_{\theta \in \Theta} \sum_{t=1}^T \norm{x(\theta; u_t) - y_t}^2 - \sum_{t=1}^T \langle 2x(\theta_t; u_t) - x(\theta_{true}; u_t) - y_t,~ x(\theta_{true}; u_t) - y_t\rangle  \\
&\leq  R_T^{pre} (\{ \theta_t \}_{t \in [T]}) + \sum_{t=1}^T \norm{x(\theta_{true}; u_t) - y_t}^2 - \sum_{t=1}^T \langle 2x(\theta_t; u_t) - x(\theta_{true}; u_t) - y_t,~ x(\theta_{true}; u_t) - y_t\rangle \\
&= R_T^{pre} (\{ \theta_t \}_{t \in [T]}) + 2\sum_{t=1}^T \langle x(\theta_{true}; u_t) - x(\theta_t; u_t),~ x(\theta_{true}; u_t) - y_t\rangle \\
&\leq R_T^{pre} (\{ \theta_t \}_{t \in [T]}) + 2\sum_{t=1}^T r_t \norm{x(\theta_{true}; u_t) - y_t}.
\end{align*}
\end{proof}

When $\frac{1}{T}R_T^{pre}(\theta_t) \rightarrow 0$, we can further simplify the range for $\frac{1}{T} \sum_{t=1}^T r_t^2 $:
\begin{align*}
    0\leq \frac{1}{T}\sum_{t=1}^T r_t^2 \leq \frac{1}{T}\sum_{t=1}^T 2 r_t \norm{x(\theta_{true}; u_t) - y_t} .
\end{align*}

\subsection{Regret performance with respect to $\ell^{sim}$ and implications}
Lastly, we discuss the regret based on $\ell^{sim}$. Recall the definition,
\begin{align*}
    R_T^{sim}(\{\theta_t\}_{t \in [T]}) & = \sum_{t=1}^T \langle \theta_t - \theta_{true}, c(y_t) - c(x(\theta_t; u_t)) \rangle  - \min_{\theta \in \Theta} \sum_{t=1}^T \langle \theta - \theta_{true}, c(y_t) - c(x(\theta_t; u_t)) \rangle .
\end{align*}

Under both perfect and imperfect information, the interpretation of $R_T^{sim}$ is not straightforward because the minimization term cannot be simplified to $0$. Recall that, under perfect information, we have shown that $R_T^{sim}$ bounds $R_T^{sub}$ and $R_T^{est}$, which themselves have natural interpretations for learning performance. Therefore, when imperfect information is present, we mainly study whether the regret bounding relations as stated in Proposition \ref{prop:pi} still hold. 

We next distinguish between two types of imperfect information: first, noise from suboptimal but feasible observations $\{y_t\}$; second, noise from potentially infeasible observations $\{y_t\}$. We observe that the key result from  Proposition \ref{prop:pi} that $R_T^{sim}$ upper bounds $R_T^{sub}$ and $R_T^{est}$ remains valid under the former suboptimality noise, but the relation may be violated for the latter case.
\begin{observation}
Suppose Assumption \ref{assp:fstr} holds, 
\begin{enumerate}
    \item[(a)] When the observations $\{y_t\}$ are feasible but possibly suboptimal, for any sequence $\{\theta_t\}_{t \in [T]}$, $R_T(\{\ell^{sim}_t\}_{t \in [T]}, \{\theta_t\}_{t \in [T]})$ upper bounds both $R_T(\{\ell^{sub}_t\}_{t \in [T]}, \{\theta_t\}_{t \in [T]})$ and $R_T(\{\ell^{est}_t\}_{t \in [T]}, \{\theta_t\}_{t \in [T]})$.
    \item[(b)] When the observations $\{y_t\}$ are possibly infeasible, for any sequence $\{\theta_t\}_{t \in [T]}$, $R_T(\{\ell^{sim}_t\}_{t \in [T]}, \{\theta_t\}_{t \in [T]})$ is not guaranteed to upper bound $R_T(\{\ell^{sub}_t\}_{t \in [T]}, \{\theta_t\}_{t \in [T]})$ or $R_T(\{\ell^{est}_t\}_{t \in [T]}, \{\theta_t\}_{t \in [T]})$.
\end{enumerate}
\end{observation}
\begin{proof}
For a given estimate sequence $\{\theta_t\}_{t \in [T]}$, from the regret definition with respect to the loss function $\ell_t^{sub}$,  we have 
\begin{align*} 
R_T^{sub}(\{\theta_t\}) & = \sum_{t=1}^T \ell_t^{sub}(\theta_t) - \min_{\theta \in \Theta} \sum_{t \in [T]}\ell_t^{sub}(\theta) \\
& = \sum_{t=1}^T \ell_t^{sim} (\theta_t) - \sum_{t=1}^T \ell_t^{est}(\theta_t) - \min_{\theta \in \Theta} \sum_{t \in [T]}\ell_t^{sub}(\theta) \\
& = \sum_{t=1}^T \ell_t^{sim} (\theta_t) - R_T^{est}(\{\theta_t\}) - \min_{\theta \in \Theta} \sum_{t \in [T]}\ell_t^{sub}(\theta) + \min_{\theta \in \Theta} \sum_{t \in [T]}\ell_t^{est}(\theta), 
\\
\text{and thus} \quad 
& R_T^{sub}(\{\theta_t\}) + R_T^{est}(\{\theta_t\}) = \sum_{t=1}^T \ell_t^{sim} (\theta_t) - \min_{\theta \in \Theta} \sum_{t \in [T]}\ell_t^{sub}(\theta) + \min_{\theta \in \Theta} \sum_{t \in [T]}\ell_t^{est}(\theta). 
\end{align*}
In the perfect information setup, in the proof of Lemma~\ref{lem:minloss}, we were able to simplify the right hand side of this expression into only $\sum_{t=1}^T \ell_t^{sim} (\theta_t)$ (which further is shown to be upper bounded by $R_T^{sim}(\{\theta_t\})$) because $\sum_{t \in [T]}\ell^{sub}_t(\theta)$ and $\sum_{t \in [T]}\ell^{est}_t(\theta)$ were both guaranteed to have a zero minimum in the absence of the noise $\epsilon_t$. 

Under a suboptimality noise, we can conclude the same bounding relation with a slightly different argument. Since we only consider feasible observations $y_t$ in this case, we still have $\ell_t^{sub}(\theta) = f(y_t; \theta, u_t) - f(x(\theta;u_t); \theta, u_t) \geq 0$ for all $t$, and thus $\min_{\theta} \sum_{t=1}^T\ell_t^{sub}(\theta)\geq0$. In addition, $\min_{\theta} \sum_{t=1}^T\ell_t^{est}(\theta) \leq \sum_{t=1}^T\ell_t^{est}(\theta_{true}) = \sum_{t=1}^T f(x(\theta_{true};u_t); \theta_{true}, u_t) - f(y_t; \theta_{true}, u_t) \leq 0$. These two relations then imply
\begin{gather*}
\begin{aligned}
R_T^{sub}(\{\theta_t\}) + R_T^{est}(\{\theta_t\}) & = \sum_{t=1}^T \ell_t^{sim} (\theta_t) - \min_{\theta \in \Theta} \sum_{t \in [T]}\ell_t^{sub}(\theta) + \min_{\theta \in \Theta} \sum_{t \in [T]}\ell_t^{est}(\theta) \\
& \leq \sum_{t=1}^T \ell_t^{sim} (\theta_t) - 0 + 0 \leq R_T^{sim}(\{\theta_t\}),
\end{aligned}
\end{gather*}
where the last inequality follows from the fact that $R_T^{sim}(\{\theta_t\})=\sum_{t=1}^T \ell_t^{sim} (\theta_t) -\min_{\theta \in \Theta} \sum_{t \in [T]}\ell_t^{sim}(\theta) \geq \sum_{t=1}^T \ell_t^{sim} (\theta_t) - \sum_{t \in [T]}\ell_t^{sim}(\theta_{true}) = \sum_{t=1}^T \ell_t^{sim} (\theta_t)$. This then implies that in the case of suboptimality loss, the regret bounds for $\ell^{sim}$ also upper bounds the regrets with respect to $\ell^{sub}$ and $\ell^{est}$ as well.

On the other hand, for the more general noisy setup with potentially infeasible $y_t$, we cannot conclude $\min_{\theta} \sum_{t=1}^T\ell_t^{sub}(\theta)\geq0$ or $\min_{\theta} \sum_{t=1}^T\ell_t^{est}(\theta)\leq0$. Consequently,  the bounding relation in Proposition~\ref{prop:pi} (c) cannot be guaranteed. 
\end{proof}

Therefore, in the imperfect information regime, without any assumption on the noisy information, in general, regret convergence with respect to $\ell^{sim}$ is not sufficient to guarantee regret convergence with respect to the other loss functions. Nevertheless, the regret bounding relations can still hold for certain type of imperfect information, such as suboptimality noise.

\section{Formulations for the Solution Oracles Used in the Implicit OL Algorithms} \label{sec:app-sol}
In this section, we give the solution oracles used in implicit OL algorithms based on $\ell^{sim}$ and $\ell^{pre}$ for two forms of agent's utility functions corresponding to the ones used in our numerical experiments. We include $\ell^{pre}$-based implicit OL in our discussion for the sake of comparison between $\ell^{sim}$-based OL framework and the previous work \cite{DongCZ2018}. In our computational study, we implemented all three of the following solution oracles that can be readily solved by standard optimization software.

\subsection{Solution Oracle for $\ell^{sim}$-based Implicit OL Algorithm}
Suppose that the squared Euclidean norm is used as the  distance generating function in the implicit OL algorithm with the solution oracle. Recall from Definition~\ref{def:lsim} that $\ell^{sim}$ has the following form under Assumption~\ref{assp:fstr}:
\[\ell^{sim}(\theta; {x}(\theta_t;{u}_t), {y}_t, {u}_t) \coloneqq
    \la \theta, c({y}_t) - c({x}(\theta_t;{u}_t)) \ra +  \la \theta_{true}, c({x}(\theta_t;{u}_t)) - c({y}_t)\ra. \]
Since the constant term in $\ell^{sim}(\theta)$ has no impact when $\ell_t^{sim}(\theta)$ is used in the objective function of an optimization problem, it can be ignored in the solution oracle formulation. Then, we deduce that the solution oracle for the $\ell^{sim}$-based implicit OL algorithm updates $\theta_{t+1}$ as
\[
\theta_{t+1} = \argmin_{\theta \in \Theta} \frac{1}{2}\norm{\theta-\theta_t}^2 + \eta_t \langle \theta,  c({y}_t) - c({x}(\theta_t;{u}_t)) \rangle.
\]
In particular, when the agent's problem has the form~\eqref{ex:prefinmkt} we have $f({x}; \theta, {u}) =\frac{1}{2}{x}^\top P{x} - \langle \theta, {x} \rangle$, i.e., $c(x)=-x$. Thus, in this case, the solution oracle for the $\ell^{sim}$-based implicit OL algorithm updates $\theta_{t+1}$ as
\[
\theta_{t+1} = \argmin_{\theta \in \Theta} \frac{1}{2}\norm{\theta-\theta_t}^2 + \eta_t \langle \theta, -y_t + x(\theta_t;u_t)\rangle.
\]
In the case of CES utility function, i.e., when the agent's problem has the form~\eqref{ex:prefces}, we have $f({x}; \theta, {u}) = \sum_{i \in [n]} (\theta)_i x_i^2$, and in this case the solution oracle for the $\ell^{sim}$-based implicit OL algorithm updates $\theta_{t+1}$ as
\[
\theta_{t+1} = \argmin_{\theta \in \Theta} \frac{1}{2}\norm{\theta-\theta_t}^2 + \eta_t \sum_{i\in [n]} \theta_i \left ((y_t)_i^2 - x(\theta_t;u_t)_i^2 \right).
\]

\subsection{Solution Oracle for $\ell^{pre}$-based Implicit OL Algorithm}
Suppose that the squared Euclidean norm is used as the  distance generating function in the implicit OL algorithm with the solution oracle. 
Then, the solution oracle for the $\ell^{pre}$-based implicit OL algorithm updates $\theta_{t+1}$ by solving the following bilevel program:
\[
\theta_{t+1} = \argmin_{\theta \in \Theta} \frac{1}{2}\norm{\theta-\theta_t}^2 + \eta_t \norm{y_t - x(\theta;u_t)}^2,
\]
where
\[
x(\theta;u_t) \in \argmin_x \left\{f(x;\theta,u_t):~ g(x;u_t)\leq0,~ x\in\mathcal{X} \right\}. 
\]

Recall that when the agent's problem has the form~\eqref{ex:prefinmkt} with a continuous polytope domain, i.e., $\mathcal{X}(u_t)=\mathcal{X}^{cp}(A_t, c_t)$, we have
\begin{gather*}
    \begin{aligned}
   {x}(\theta; u_t) \coloneqq
   \argmax_{{x}} \left\{ -\frac{1}{2}{x}^\top P{x} + \langle \theta, {x} \rangle :~
   A_t x \leq c_t, ~ x\in\R^n_+ 
   \right\},
    \end{aligned} 
\end{gather*}
where $P \in \Se_{++}^n$ is a fixed positive definite matrix known by both the learner and the agent. Using the KKT optimality conditions for the inner problem, and then introducing binary variables to linearize the resulting nonlinear relations, it is possible to reformulate this bilevel problem into a single level optimization problem with binary variables. In particular, in this case, following these outlined steps,  \cite{DongCZ2018} proposed the following reformulation of this bilevel problem into a single level MISOCP:
\begin{gather*}
\begin{aligned}
\theta_{t+1} = \argmin_{\theta, x,w,v, y, z} \quad & \frac{1}{2} \norm{\theta - \theta_t}^2 + \eta_t \norm{y_t-x}^2 \\
\text{s.t. } & A_t x \leq c_t,~ x\in\R^n_+ \\
& w_i \leq M y_i,~i\in[n] \\
& -x_i \geq -M(1-y_i) ,~i\in[n] \\
& v_j \leq M z_j,~j\in[m] \\
& (A_t)_j^\top x - (c_t)_j \geq -M(1-z_j),~j\in[m] \\
& P x - \theta + A_t^\top v - w= 0 \\
& v \in \R_+^m, w \in \R_+^n,~y \in \{0,1\}^n,~ z \in \{0,1\}^m \\
& \theta \in \Theta.
\end{aligned}
\end{gather*}
Here, $M$ is the so-called big-$M$ constant. The variables $v\in\R^m_+, w \in \R_+^n$ are the variables corresponding to the Lagrangian multipliers, the binary variables $y_i \in \{0,1\}$ for all $i \in [n]$ are used to linearize the KKT condition $w_i x_i = 0$, and $z_j\in\{0,1\}$ for all $j\in[m]$ are introduced to linearize the KKT relation $v_j((A_t)_j^\top x - (c_t)_j)=0$. Therefore, the big-$M$ constants must be selected so that they upper bound the components in the bilinear expressions, e.g., $x_i$ and $w_i$ for the complementarity constraint $w_ix_i=0$ as well as $(A_t)_j^\top x - (c_t)_j$ and $v_j$ for the constraint $v_j((A_t)_j^\top x - (c_t)_j)=0$. Because in our instances the agent's domain for $x$ is bounded, we can easily obtain bounds on $x_i$ and $(A_t)_j^\top x - (c_t)_j$ terms. It is also possible to derive an upper bound for the Lagrange multipliers under a Slater condition assumption on the primal problem. Nevertheless, it is well known that using big-$M$ formulations significantly degrade the optimization solver performance, and instead it is encouraged in Gurobi solver that such big-$M$ constraints are encoded as indicator constraints, which is a form of logical constraints supported by Gurobi. In our experiments, we follow this approach and use the indicator constraint feature of the Gurobi solver. Note that this alternative implementation is possible because the big-$M$ constraints essentially represent a complementarity type logical condition.

Note that the continuous knapsack domain $\mathcal{X}^{ck}(p_t,b_t)$ is a special case of the continuous polytope domain $\mathcal{X}^{cp}(A_t, c_t)$, and thus the same reformulation also holds in that case.

Finally note that when the agent's problem has the form~\eqref{ex:prefces} with an equally constrained knapsack domain, i.e., $\mathcal{X}(u_t)=\mathcal{X}^{eck}(p_t, b_t)$, we have
\begin{gather*}
    \begin{aligned}
   {x}(\theta; u_t) \coloneqq
   \argmin_{{x}} \left \{\sum_{i \in [n]} \theta_i x_i^2
    :~ p_t^\top x = b_t,~ x\in\R^n_+
    \right\}.
    \end{aligned} 
\end{gather*}
In this case, the bilevel program corresponding to the solution oracle in the $\ell^{pre}$-based implicit OL algorithm has the following single level reformulation.
\begin{gather*}
\begin{aligned}
\theta_{t+1} = \argmin_{\theta, x, w, v, y } \quad & \frac{1}{2} \norm{\theta - \theta_t}^2 + \eta_t \norm{y_t-x}^2 \\
\text{s.t. } & p_t x = b_t,~ x\in\R^n_+ \\
& w_i \leq M y_i,~i\in[n] \\
& -x_i \geq -M(1-y_i) ,~i\in[n] \\
& 2\theta_i x_i + v (p_t)_i - w_i= 0, ~i \in [n] \\
& v \in \R, w \in \R_+^n,~y \in \{0,1\}^n \\
&\theta \in \Theta. 
\end{aligned}
\end{gather*}

Unfortunately, this nonconvex mixed integer program contains the bilinear terms $\theta_i x_i$, where both $x$ and $\theta$ are continuous variables, in a general constraint, not of a complementarity type constraint. Note that the primal domain is equality constrained continuous knapsack, and thus we can find an upper bound on $x$ variables. Moreover, for $\theta\in\Theta$ and when $\Theta$ is bounded like the Euclidean ball or the simplex case that we focus on in this paper, we can find a bound on $\theta$ as well. However, because this bilinear term of $\theta_i x_i$ is appearing in a general constraint and not in a complementary constraint, there is no technique to reformulate this nonconvexity as linear constraints by introducing new binary variables. Hence, in this case the $\ell^{pre}$-based implicit OL algorithm requires a computationally expensive general purpose nonconvex solution oracle.

\end{document}

%% file: fatma-packages.tex
\usepackage[utf8]{inputenc} 
\usepackage[T1]{fontenc}    
\usepackage[english]{babel}

\usepackage{amsfonts}
\usepackage{nicefrac}       
\usepackage{microtype}      
\usepackage{lmodern}
\usepackage{amssymb,amsmath,amsthm}
\usepackage{stmaryrd}
\usepackage{bbm,bm}
\usepackage{latexsym}
\usepackage{xcolor}

\usepackage{enumerate}
\usepackage{verbatim}
\usepackage{booktabs}       

\usepackage{url}            
\usepackage[colorlinks=true]{hyperref}
\usepackage[numbers,sort&compress,square,comma]{natbib}
\usepackage[capitalise,noabbrev]{cleveref}

\usepackage{parskip}
\usepackage{geometry}
\usepackage[short]{optidef}
\usepackage{algorithmic,algorithm}
\usepackage{authblk}

\usepackage{thmtools}

\declaretheorem{theorem}
\declaretheorem{corollary}
\declaretheorem{lemma}
\declaretheorem{proposition}
\declaretheorem{observation}

\declaretheoremstyle[qed=$\square$]{definitionwithend}
\declaretheorem[style=definitionwithend]{definition}
\declaretheorem[style=definitionwithend]{assumption}
\declaretheorem[style=definitionwithend]{example}
\declaretheorem[style=definitionwithend]{remark}

\crefname{fact}{Fact}{Facts}

\usepackage{import}
\usepackage{pdfpages}
\usepackage{transparent}

%% file: fatma-macros.tex
\def\cbl{\color{blue}}

\definecolor{gold}{rgb}{0.85,0.65,0}

\newcommand{\norm}[1]{\ensuremath{\left\lVert #1 \right\rVert}}

\let\emptyset\varnothing

\def\R{{\mathbb{R}}}

\DeclareMathOperator*{\argmin}{arg\,min}
\DeclareMathOperator*{\argmax}{arg\,max}

\DeclareMathOperator*{\E}{\mathbb{E}}

\def\Se{{\mathbb{S}}}
\def\CX{{\cal X}}
\def\ra{\rangle}
\def\la{\langle}
\newcommand{\dgf}{d.g.f.}
\DeclareMathOperator{\Prox}{Prox}